%% file: main.tex
\documentclass[english]{article}
\usepackage[utf8]{inputenc}
\usepackage[T1]{fontenc}
\usepackage{graphicx} 
\usepackage{amsmath,amssymb,amsfonts}%
\usepackage[margin=3cm]{geometry}
\usepackage{amsthm}
\usepackage{dsfont}
\usepackage{babel}
\usepackage{subcaption}
\usepackage[linesnumbered,ruled,vlined]{algorithm2e}
\usepackage{xcolor}
\usepackage{csquotes}
\usepackage{authblk}
\usepackage{esvect}
\usepackage{mathtools}
\usepackage{xparse}

\title{From entropic transport to martingale transport, and applications to model calibration}

\author[1]{Jean-David Benamou}
\author[1,2]{Guillaume Chazareix}
\author[2]{Grégoire Loeper}

\affil[1]{INRIA Paris}
\affil[2]{BNP Paribas Global Markets}

\NewDocumentCommand{\expect}{ e{_} e{^} s o >{\SplitArgument{1}{|}}m }{%
  \operatorname{\mathbb{E}}
  \IfValueT{#1}{{\!}_{#1}}
  \IfValueT{#2}{{\!}^{#2}}
  \IfBooleanTF{#3}{
    \expectarg*{\expectvar#5}%
  }{
    \IfNoValueTF{#4}{
      \expectarg{\expectvar#5}%
    }{
      \expectarg[#4]{\expectvar#5}%
    }%
  }%
}
\NewDocumentCommand{\expectvar}{mm}{%
  #1\IfValueT{#2}{\nonscript\;\delimsize\vert\nonscript\;#2}%
}
\DeclarePairedDelimiterX{\expectarg}[1]{[}{]}{#1}

\newcommand*\diff{\mathop{}\!\mathrm{d}}

\usepackage[backend=bibtex,style=numeric,doi=false,eprint=false,url=false]{biblatex}
\DefineBibliographyStrings{english}{
  andothers = {et\addabbrvspace al\adddot},
}

\newcommand{\as}{\text{a.s.}}
 
    \newcommand*{\eqdef}{\stackrel{\mathrm{def}}{=}}
\newcommand{\SP}[3]{\langle #1  ,  #2   \rangle_{#3} }

\newcommand{\RR}{{W}}

\NewDocumentCommand{\dint}{e{_} e{^}}{%
  {\displaystyle \int\IfValueT{#1}{_{#1}}\IfValueT{#2}{^{#2}}}
}

\newcommand{\bs}{{\overline{\sigma}}}
\newcommand{\E}{\mathbb{E}}
\newcommand{\F}{\mathcal{F}}
\newcommand{\PP}{\mathbb{P}}

\NewDocumentCommand{\bPP}{e{_} e{^}}{%
  {\overline{\PP\IfValueT{#1}{_{#1}}\IfValueT{#2}{^{#2}}}}
}

\newcommand{\R}{\mathbb{R}}

\newcommand{\Pc}{\mathcal{P}}
\newcommand{\Mc}{\mathcal{M}}
\newcommand{\Gc}{\mathcal{G}}
\newcommand{\Kc}{\mathcal{K}}
\newcommand{\Fc}{\mathcal{F}}
\newcommand{\Ic}{\mathcal{I}}
\newcommand{\Sc}{\mathcal{S}}

\newcommand{\Xc}{\mathcal{X}}

\newcommand{\C}{\mathcal{C}}

\renewcommand{\P}{\mathbf{P}}
\DeclareMathOperator{\KL}{KL}

\newtheorem{proposition}{Proposition}[section]
\newtheorem{definition}{Definition}[section]
\newtheorem{theorem}{Theorem}[section]
\newtheorem{remark}{Remark}[section]
\newtheorem{lemma}{Lemma}[section]

\begin{document}

\maketitle
\begin{abstract}
  We propose a discrete time formulation of the semi martingale optimal transport problem based on multi-marginal entropic transport.
  This approach offers a new way to formulate and solve numerically
  the calibration problem proposed by \cite{guo2021optimal}, using a
  multi-marginal extension of Sinkhorn algorithm as in \cite{BenamouE,CarlierGFE,BenamouMFG2}.
  In the limit when the time step goes to zero we recover, as detailed in the companion paper \cite{Mpaper},  a semi-martingale process, solution to a semi-martingale optimal transport problem, with a cost function involving the so-called specific entropy introduced in
  \cite{Gantert91}, see also \cite{follmer22} and \cite{backhoff23}.

\end{abstract}

\input{intro.tex}

\input{continuous.tex}

\input{mmmt.tex}

\input{duality.tex}

\input{applications.tex}

\input{numerics.tex}
\printbibliography

\appendix

\input{appendix.tex}
\end{document}

%% file: intro.tex
\section{Introduction}
Applications of Semi Martingale Optimal Transport (SMOT) in finance have been
the object of several recent studies (\cite{Tan_2013}, \cite{guo2021optimal}, \cite{Guyon_2022} amongst others). This framework is particularly well adapted  to the problem of model calibration: Find a diffusion model that is compatible with observed option prices.
SMOT is the stochastic version of
Dynamic Optimal Transport (DOT), that was  introduced by \cite{BBA}, as a generalization of OT  where transport is achieved by a time dependent flow
minimizing the kinetic energy.

While the theoretical aspects of these problems are now well understood, the numerical implementation remains challenging.

In the meantime, a stochastic relaxation of static (i.e. not dynamic) optimal transport, known as {\it Entropic} Optimal Transport (EOT), has shown to be solvable very efficiently, by the so-called {\it Sinkhorn} algorithm (see \cite{PeyreB} for a review). Interestingly,  while there is equivalence between the Static OT problem and its dynamic version, the Entropic regularisation of OT can also be seen either as  a static problem, or as variant of the DOT adding a constant volatility diffusion to the governing model, this dynamic problem is known as the Schr\"odinger problem, see \cite{LeonardS}. 

From a  mathematical perspective, all problems (DOT, SMOT and EOT) can  be seen as a variant of the same problem: find a process described by the SDE
\[
    \diff X_t = \mu_t \, \diff t + \sigma_t \, {\diff W}_t,
\]
the induced probability $\PP$ on the  space of continuous paths,
with distribution constraints of the form of moment constraints $\expect_\PP{\psi_i(X_{t_i})}=c_i$\footnote{Note that prescribing the whole distribution is equivalent to prescribing enough moment constraints} , and minimising a Lagrangian
\begin{align}\label{eq:SMOT}
    \expect_\PP*{\dint_0^T F(\mu_t,\sigma_t^2) \diff t} := {\cal F}(\PP).
\end{align}

Classical DOT  \cite{BBA} corresponds to the particular case $F=|\mu|^2$ if $\sigma \equiv 0$, $+\infty$ otherwise.

The Semi-Martingale Transport will handle general forms of $F$, as long as it is convex with respect to $(\mu,\sigma^2)$. \\

Classical Entropic OT  \cite{LeonardS} corresponds to the particular case $F=|\mu|^2$ if $\sigma \equiv \bs$ where $\bs$ is constant, $+\infty$ otherwise.
It is also one of the formulations of the {\it Schr\"odinger's} problem , i.e. minimizing the relative entropy (aka Kullblack Leibler divergence) of $\PP$:  $\KL(\PP | \RR_\bs)$, with respect to the  Wiener measure $\RR_\bs$ (with a given constant volatility $\bs$), under initial and terminal conditions on the law of $X_{t=0}$ and $X_{t=T}$. Thanks to the  properties of the relative entropy, Classical EOT in its static formulation can be solved very efficiently by the Sinkhorn's algorithm.

A general approach (see \cite{benamou2021} for a review) is a time discretization that leads to a so-called Multi-Marginal OT problem . In this setting the minimization is performed over the law $\PP^h$ ($h$ is the time step) of a vector-valued random variable whose marginals represent densities at each time step.
In this paper, we  use this time discretization method and an
entropic penalization  to solve problem \eqref{eq:SMOT}:
We minimize  the sum  of a discretized form of $\F$ in
\eqref{eq:SMOT}, $\F_h(\PP^h)$
(by taking the obvious discrete versions of $\mu$ and $\sigma$)
and the discrete time relative entropy regularization $\KL(\PP^h | \RR^h_\bs)$.

The drawback of minimising an energy in the form $\KL(\PP | \RR_\bs)$
is that by essence the minimizer is constrained at $\sigma=\bs$ and cannot satisfy constraints on $\mu$ (for instance $\mu\equiv 0$ or $\mu=r$, the interest rate) familiar in finance, since $\mu$ is precisely the only degree of freedom used to comply with  the distribution constraints.
We propose   to overcome this issue by
considering a proper scaling of the discrete relative entropy and its convergence property :
if $\PP^h$ a sequence of Markov chains converging to the law of a diffusion $\PP$ with (possibly local) drift $\mu$ and  volatility
$\sigma$ then
\begin{align}\label{specrelent}
    \lim_{h\searrow 0 }  h\, \KL(\PP^h | \RR^h_\bs)   = \expect_\PP*{\dint_0^T \frac{\sigma^2}{\bs^2} - 1 - \log{\frac{\sigma^2}{\bs^2}} \, \diff t} =:  \Sc(\PP |\RR_\bs ),
\end{align}
or in short $\Sc(\sigma |\bar\sigma)$.

The "specific relative entropy" $\Sc$ defined above has been introduced in
\cite{Gantert91} see also  \cite{follmer22} \cite{backhoff23}.

It is shown in \cite{Mpaper} that minimizers of $ \F_h(\PP^h) + h\, KL(\PP^h | \RR^h_\bs)  $ converge in the limit $h\searrow 0$ to a diffusion process $\PP$   minimizing the modified cost
\[
    {\F}(\mu,\sigma^2) +\Sc(\PP | \RR_\bs).
\]
For $h >0$ we recover a  Multi-Marginal EOT, and a discrete Markov chain that can still be used for simulations.

The interest of this approach is not only theoretical: classical methods to solve \eqref{eq:SMOT} involve maximizing the dual problem through gradient ascent or primal-dual approaches. These methods imply  solving a fully non-linear Hamilton-Jacobi-Bellman equation at each iteration (\cite{BBA}, \cite{PapaA}, \cite{LoeperA}).

Our approach by Multi-Marginal Sinkhorn's algorithm, extending \cite{BenamouE} and \cite{BenamouMFG2},  computes the same object
with the usual convergence guarantees of classical EOT   \cite{CarlierMM} \cite{DimarinoG}.

This paper describes the dual formulation of the problem in the context of local volatility calibration,   the associated Sinkhorn algorithm and its practical implementation, with numerical examples.

%% file: continuous.tex
\section{Martingale Optimal Transport for model calibration}
\label{sec:continuous}
The continuous formulation of (Semi-)Martingale Optimal Transport was introduced in \cite{Tan_2013}, and extended for multiple calibration applications as presented in the survey \cite{guo2021optimal}. We are interested here in the one-dimensional formulation of this problem, for the calibration of a local volatility in space and time using a finite number of discrete constraints. \\

Let $\Omega = C([0, T], \R), \, T > 0$ be the set of continuous paths, and $\Pc$ the set (or a convex subset of) probability measures on $\Omega$.
\\
The input of the calibration problem is a set of  discrete constraints indexed by $i \in \Ic := \{1, \dots, N_c\}$ described by $(\tau_i, c_i, G_i)_{i \in \Ic}$ where for each $i$ the  triplet $(\tau_i, c_i, G_i)$ is maturity, price and payoff function of an observed derivative price on the market.
\\
We will seek for an element $\PP \in \Pc$ such that
\[
    \expect_{\PP}{G_i(X_{\tau_i})} = c_i.
\]
As an example, calibrating a set of call options at a fixed  maturity $T$ would lead to $G_i(x) = (x - K_i)^+$, where $K_i$ is the strike  of the $i$-th option, and $\tau_i = T$ for all $i$.
Moreover, we assume that there are only a finite
set of maturities $\tau_i$ and thus the set $\Ic$ can be partitioned as $\Ic = \bigcup_k \Ic_k, \, \Ic_k := \{i \in \Ic, \tau_i = t_k\}$, where $t_k$ is the $k$-th distinct maturity in the set of constraints.
\\
To formulate the problem as a constrained minimization problem on $\Pc$, we restrain our search to the set
${\cal P}^0_s\subset {\cal P}$ such that, for each $\PP\in{\cal P}^0$, $X\in \Omega$ is an $({\cal F},\PP)$-semimartingale on $[0,1]$ given by
\[
    X_t=X_0+A^\PP_t + M^\PP_t,\quad \langle X\rangle_t=\langle M^\PP_t\rangle=B^\PP, \quad \PP\text{-}\as, \quad t\in[0,1],
\]
where $M^\PP$ is
an $({\cal F},\PP)$-martingale on $[0,1]$ and $(A^\PP,B^\PP)$ is ${\cal F}$-adapted and $\PP$-$\as$ absolutely continuous with respect to time. In particular, $\PP$ is said to be have characteristics $(\mu,\sigma^2)(\PP)$, which are defined in the following way,
\[
    \quad \mu =\frac{\diff A^\PP_t}{\diff t}, \sigma^2_t=\frac{\diff B^\PP_t}{\diff t},
\]
Note that $(\mu,\sigma^2)$ is ${\cal F}$-adapted and determined up to $d\P\times dt$, almost everywhere.
We now let $F(t,x,a,b)$ be convex with respect to $(a,b)$ for every $(t,x)$, and seek for
\begin{equation}
    \label{eq:cont_mot}
    \tag{CMOT}
    \mathcal{V} = \inf_{\PP \in \Pc} \E_\PP \int_0^T F(t, X_t, \mu, \sigma^2) \, \diff t,
\end{equation}
In the calibration case we alsos impose $X_0 = x_0 \in \R$ (i.e. $X_0 \sim \delta_{x_0}$) as the
derivative price is known at time $0$.

At this stage, the processes $\mu, \sigma$ are very general and can be generally path-dependent, however, as showed in \cite{guoCalibrationLocalStochasticVolatility2021a}, they can be chosen as local processes, i.e. functions of $(t,X_t)$ only:
indeed, for any choice of $\mu,\sigma$, there exists a local version $\mu(t,x),\sigma(t,x)$ that preserves the constraints (i.e. option prices) and that can only reduce the cost \eqref{eq:cont_mot}.
The minimization problem can therefore be reduced to
$X_t$ solutions of the stochastic differential equation:
\begin{equation}
    \label{eq:cmot_sde}
    \tag{SDE}
    \diff X_t = \mu(X_t, t) \, \diff t + \sigma(X_t, t) \, {\diff B}_t,
\end{equation}
where $B_t$ is a standard Brownian motion.\\

The function $F$  can be decomposed into  a sum of model  constraints $F_{mc}$, calibration constraints $F_{cc}$ and regularization $F_{r}$ components:
\begin{itemize}
    \item Model constraints: for instance, if we want to impose that the underlying $X_t$ follows a pure diffusion model, i.e. $\mu \equiv 0$
          this can be imposed by choosing $F_{mc}$ as :
          \begin{equation*}
              F_{mc}( \mu) = \begin{cases}
                  0       & \text{if } \mu = 0 \\
                  +\infty & \text{otherwise}.
              \end{cases}
          \end{equation*}
          (or a soft version).

    \item Calibration constraints expressed as
          \begin{equation*}
              F_{cc}(X_t ) = \begin{cases}
                  0       & \text{if }  \expect{G_i(X_{\tau_i})} = c_i . \\
                  +\infty & \text{otherwise}.
              \end{cases}
          \end{equation*}
          (or a penalisation of the constraint).
    \item Regularization/model assumptions:  $F_r$ helps  enforce qualitative properties of the model. We might want for instance $\sigma$  to be close to a prescribed guess $\bs$, which can be enforced by choosing $F$ as a penalty function of the form $F(\sigma^2/\bar\sigma^2)$, for instance Loeper \cite{loeper2016option} uses 
    \[
    F= (\sigma-\bar\sigma)^2
    \](which is convex in $\sigma^2$ and is linked to the Bass martingale problem), or 
          in  \cite{guo2021optimal} they use
          \begin{equation}
              \label{GR}
              F_r(\sigma^2) = a \left(\frac{\sigma^2}{\bs^2}\right)^p + b \left(\frac{\sigma^2}{\bs^2}\right)^{-q} + c
          \end{equation}
          with $p, q, a, b > 0$ and $c \in \R$ such that $F$ is convex with minimum  at $\sigma = \bs$. It is also a barrier as $\sigma^2$ goes to $0$.

          In this paper   we choose
          \begin{equation*}
              F_r( \sigma^2) =  \Sc ( \sigma  | \bs).
          \end{equation*}
          with $\Sc$ defined in \eqref{specrelent}.
          As explained in the introduction it allows to discretize (in time)  the problem as
          a multi-marginal EOT problem.
          It is again convex with minimum  at $\sigma = \bs $ and a
          barrrier as $\sigma^2$ goes to $0$ but unlike (\ref{GR}) it is sublinear
          as $\sigma \nearrow + \infty$. This difficulty is discussed in \cite{Mpaper}.
\end{itemize}

%% file: mmmt.tex
\section{Discretisation into a Multi-Marginal Martingale Transport}

\subsection{Notations}

We will discretize our problem in time, replacing the interval $[0, T]$ with a regular grid of $N_T + 1$ timesteps $t_k = k\,h$ for $k \in \{0, \dots, N_T\} =: \Kc^h$, where $h := T/N_T$ is the time step.
We impose that all the calibration times $\tau_i$ are included in the grid, i.e. $\tau_i = t_{k_i}$ for some $k_i \in \Kc^h$.

Instead of functions $t \mapsto \omega(t)$, we consider their discrete path counterparts, which are n-tuples $(\omega_0, \dots, \omega_{N_T}) \in \R^{N_T+1}$ for $k \in \Kc^h$, in which $\omega_k$ corresponds to the value of the path at time $t_k$. Instead of $\R^{N_T}$, we denote by $\Xc_k$ the space of values that $\omega_k$ can take, and by $\Omega^h := \Pi_{i=0}^{N_T} \Xc_i$ the space of discrete paths.

An element $(t \mapsto \omega(t)) \in \Omega$ is hence replaced by a n-tuple $(\omega_0, \dots, \omega_{N_T}) \in \Omega^h$ with $\omega_k \in \Xc_k$ for $k \in \Kc^h$.

We are hence searching for a probability measure $\PP^h$ on $\Omega^h$.
We denote $(X_k)_{k \in \Kc^h}$ the canonical process of $\PP^h$ on $\Omega^h$. We will denote by $\PP^h_k := {X_k}\#\PP^h$ the marginal law of $\PP^h$ at timestep $k \in \Kc^h$, and by $\PP^h_{k,l} := (X_k, X_l)\#\PP^h$ the joint law of time steps $k \in \Kc^h$ and $l \in \Kc^h$.

We note $\Kc^h_{-i} = {k \in \Kc^h\setminus \{i\}}$ the set of timesteps except timestep $i$, and $\diff x_{-i} = \prod_{\Kc^h_{-i}} \diff x_k$, which allows to write the marginal law $\PP^h_{k}$ as $\PP^h_{k} = \dint \PP^h(x_k, \diff x_{-k})$ and joint laws in a similar fashion.
Similarly, we note $\diff x_{[i, j]} = \Pi_{k=i}^j \diff x_k$.

We note $\rho_0$ the initial marginal of our process, which is imposed, $X_0 \sim \rho_0$. It may or may not be a Dirac in our case.

We denote by $\bPP^h$ the reference measure on $\Omega^h$ that we will use to regularize the problem. We will denote by $(Y_k)_{k \in \Kc^h}$ the canonical process of $\bPP^h$ on $\Omega^h$. It's law is determined by a Euler-Maruyama discretisation of the continuous reference process :
\[
    Y_{k+1} = Y_k + \overline{\mu}(Y_k, kh) \, h + \overline{\sigma}(Y_k, kh) \, h^{1/2} \, Z_k, \quad \forall k \in \Kc^h_{-0}, \, Y_0 \sim \rho_0.
\]
We write $\Pc^h_{\text{EM}}$ the set of probability measures on $\Omega^h$ whose canonical process $Y_k$ can be writen as a Euler-Maruyama discretisation as such, with $\mu_k = \mu(Y_k, kh)$ and $\sigma_k = \sigma(Y_k, kh)$.

For any probability measure $\PP^h \in \Pc(\Omega^h)$ and $\bPP^h \in \Pc(\Omega^h)$, we note $\KL(\PP^h|\bPP) = \expect_{\PP^h}*{\log\left(\frac{\diff \PP^h}{\diff \bPP^h}\right) - 1}$ the
Kullback-Leibler divergence between $\PP^h$ and $\bPP^h$ if $\PP^h \ll \bPP^h$. By convention, if $\PP^h \not\ll \bPP^h$, we set $\KL(\PP^h|\bPP^h) = +\infty$.

For two continuous diffusion processes $\PP$ with volatility $\sigma^2$ and $\bPP$ with volatility $\overline{\sigma}^2$, we note $\Sc(\sigma^2| \bs^2) = \expect_\PP*{\dint \frac{\sigma^2}{\overline{\sigma}^2} - 1 - \log\left(\frac{\sigma^2}{\overline{\sigma}^2}\right) \, \diff t}$ the specific entropy between $\PP$ and $\bPP$.
For any two discrete probabilities in $\Pc^h_{\text{EM}}$, we will denote $\Sc^h(\PP^h|\bPP^h) = \expect_{\PP^h}*{h \sum_{k = 0}^{N_T} \frac{\sigma_k^2}{\overline{\sigma_k}^2} - 1 - \log\left(\frac{\sigma_k^2}{\overline{\sigma_k}^2}\right)}$ the discrete specific entropy between $\PP^h$ and $\bPP^h$, which is a Riemann sum discretizing the continuous specific entropy.

\subsection{Discrete drifts and diffusions coefficients}
\label{sec:prelim}
As opposed to the continuous-time approach, which uses Markovian projections of the processes, and as such the variable being optimised are functions representing the drift and volatility, in this discrete-time approach, we will directly optimise on $\PP^h$. 
In order to  justify the choice of moment variables in the discrete problem, we first consider Euler-Maruyama discretization of a diffusion process. Let a diffusion process $X_t$ with drift $\mu$ and volatility $\sigma$, following the SDE :
\[
    \diff X_t = \mu(X_t, t) \, \diff t + \sigma(X_t, t) \, {\diff \RR}_t.
\]

Consider the Euler-Maruyama time  discretization of the process:
\[
    X^h_{k+1} = X^h_{k} + \mu(X^h_k, kh) \, h + \sigma(X^h_k, kh) \, h^{1/2} \, Z_k
\]
where $\forall k \in \{0, \dots, N_T\} := \Kc^h, \, Z_k$ is a standard normal random variable, of which we note $\PP^h$ the law. We have $\PP^h \in \Pc^h_{\text{EM}}$.

For such a process, we can compute the following quantities from conditional expectations :
\begin{align}
    \label{beta_def} \beta_k(x)   & = \frac{1}{h} \expect*{X_{k+1} - X_k | X_k = x} = \mu(x, kh),                                                                      \\
    \label{alpha_def} \alpha_k(x) & = \frac{1}{h} \expect*{(X_{k+1} - X_k)^2 | X_k = x} = \mu(x, kh)^2 \, h + \sigma^2(x, kh) \xrightarrow[h \to 0]{} \sigma^2(x, kh).
\end{align}
These variables are computed from the law $\PP^h$ and are the discrete counterpart of the drift and volatility of the continuous process. They can hence be used to compute the discretisation of $F(t, X_t, \mu, \sigma^2)$.

A more general framework is to consider a variable $b_k : \Xc \to \R^K$ defined by taking the conditional expectation of a general function $B : (\Xc, \Xc) \to \R^K$ depending on two consecutive timesteps :
\[
    b_k(x) = \frac{1}{h} \expect{B(X_{k}, X_{k+1}) | X_k=x}.
\]

The vector formed with variables $\beta_k$ and $\alpha_k$ can be computed as such by using the function $B(X, Y) = \begin{bmatrix}
        (Y-X) \\ (Y-X)^2
    \end{bmatrix}$.

We might alternatively want to control other types of moments, such as the skewness or the kurtosis of the process. For this reason, we will consider a general function $B$ and the corresponding $b_k$ in the following, and not specifically $\beta_k$ and $\alpha_k$.

\subsection{Specific relative entropy}
\label{sec:specentr}
We give a formal derivation of the Specific Entropy (see
\cite{Gantert91}   \cite{follmer22} \cite{backhoff23}).\\

The Kullback-Leibler divergence between two normal laws $\mathcal{N}(\mu_1, \sigma_1^2)$ and $\mathcal{N}(\mu_2, \sigma_2^2)$ is equal to :
\begin{equation*}
    \KL(\mathcal{N}(\mu_1, \sigma_1^2) | \mathcal{N}(\mu_2, \sigma_2^2)) = \frac{1}{2}\left(\frac{\sigma_1^2 + (\mu_1 - \mu_2)^2}{\sigma_2^2} - 1 - \log\left(\frac{\sigma_1^2}{\sigma_2^2}\right)\right).
\end{equation*}

Consider two diffusion measures $\PP$ and $\bPP$ defined by the following SDEs on their respective canonical processes $X$ and $Y$ :
\begin{align}
    \diff X_t & = \mu(X_t, t) \, \diff t + \sigma(X_t, t)  \, {\diff \RR}_t, \,  X_0 \sim \rho_0,                         \\[8pt]
    \diff Y_t & = \overline{\mu} (Y_t, t) \,  \diff t + \overline{\sigma}(Y_t, t) \, {\diff \RR}_t,  \, Y_0 \sim \rho_0 .
\end{align}
We can discretize on a grid of step $h$ as law $\PP^h$ and $\bPP^h$ using the Euler-Maruyama discretization from previous section, giving their respective canonical processes $X^h$ and $Y^h$ :
\begin{align}
    X^h_{k+1} & = X^h_{k} + \mu(X^h_k, kh) \,  h + \sigma(X^h_k, kh) \,  h^{1/2}  \, Z_k, \quad \forall k \in \Kc^h, \,
    X^h_0 \sim \rho_0     ,                                                                                                                                \\[8pt]
    Y^h_{k+1} & = Y^h_{k} + \overline{\mu}(Y^h_k, kh) \, h + \overline{\sigma}(Y^h_k, kh) \, h^{1/2} \, Z_k, \quad \forall k \in \Kc^h, Y^h_0 \sim \rho_0.
\end{align}
Noting $\mu_k(x) = \mu(x, kh)$, $\sigma_k(x) = \sigma(x, kh)$, and $\overline{\mu}_k(x) = \overline{\mu}(x, kh)$ and $\overline{\sigma}_k(x) = \overline{\sigma}(x, kh)$, we can conclude the transitions laws $\PP^h_{k,k+1}$ and $\bPP^h_{k, k+1}$ are normal laws, and the Kullback-Leibler divergence can then be decomposed as follows :
\begin{equation}
    \label{SRE}
    \begin{array}{ll}

        h \, \KL(\PP^h|\bPP^h) & = h\, \sum_{k=0}^{N_T - 1} \dint \KL(\mathcal{N}(\mu_k(x) h, \sigma_k(x)^2 h) | \bPP^h_{k,k+1}) \,  \rho_k(dx)                                                                                                                          \\[10pt]
                               & = \frac{h}{2}  \, \sum_{k=0}^{N_T - 1} \dint\left(\frac{\sigma_k(x)^2 h + ((\mu_k(x)-\overline{\mu}_k(x))h)^2}{\overline{\sigma}_k(x)^2h} - 1 - \log\left(\frac{\sigma_k(x)^2 h}{\overline{\sigma}_k(x)^2h}\right)\right) \, \rho_k(dx) \\[10pt]
                               & = \frac{h}{2} \,  \sum_{k=0}^{N_T - 1} \dint\left(\frac{\sigma_k(x)^2 + (\mu_k(x) - \overline{\mu}_k(x))^2h}{\overline{\sigma}_k(x)^2} - 1 - \log\left(\frac{\sigma_k(x)^2}{\overline{\sigma}_k(x)^2}\right)\right) \, \rho_k(dx)       \\[10pt]
                               & \xrightarrow[h \to 0]{} \mathcal{S}(\sigma^2| \bs^2) := \frac{1}{2} \dint \expect_{\PP^h}*{\frac{\sigma^2}{\overline{\sigma}^2} - 1 - \log\left(\frac{\sigma^2}{\overline{\sigma}^2}\right)} \diff t
    \end{array}
\end{equation}
More generally, it is shown in  \cite{Mpaper} that
the discrete
Kullback-Leibler Divergence allows to control the approximation of the
volatility of the discrete process (\ref{alpha_def}).
This  motivates to use the specific entropy as a regulariser of the continuous problem, because it is linked to a natural discretization in terms of the Kullback-Leibler divergence, and entropy-regularized optimal transport is a thoroughly studied problem.

\subsection{Discretisation}
\label{sec:discretisation}
We will now use the previously defined tools to discretize the continuous problem (\ref{eq:cont_mot}) in time, in the case of a specific entropy regularizer, since it has a natural discretization in terms of the Kullback-Leibler divergence.

We first directly discretize (\ref{eq:cont_mot}) in time as a Riemann sum  and using (\ref{SRE}) and the variables $\beta_k$ and $\alpha_k$ defined in equations (\ref{beta_def}) and (\ref{alpha_def}), we obtain :
\begin{equation*}
    \expect_{\PP}*{\dint_0^T (F+\Sc)(t, X_t, \mu, \sigma^2) \diff t} \approx h \,  \expect_{\PP^h}*{\sum_{k=0}^{N_T} F(t, X_k, \beta_k(X_k), \alpha_k(X_k))} + h  \, \KL(\PP^h|\bPP^h)
\end{equation*}
where $\mathcal{S}$ the specific entropy introduced in section \ref{sec:specentr}.

We can hence formulate our discretization of Problem \ref{eq:cont_mot} as :

\begin{equation}
    \label{mmmtprim}
    \inf_{\substack{\PP^h \in \Pc^h}} h  \,  \expect_{\PP^h}*{\sum_{k=0}^{N_T} F(t, X_k, \beta(X_k), \alpha(X_k))} + h \KL(\PP^h|\bPP^h)
\end{equation}
where $\Pc^h$ is the set of probability measures on n-uplets respecting the constraints, that is :
\[
    \Pc^h = \{\PP^h \in \Pc(\mathbb{R}^{N_T}), \text{ s.t. } X_0\#\PP^h = \mu_0 \text{ and } \forall i \in \Ic,\, \E_{\PP^h}[G_i(X_{\tau_i})] = c_i\}
\]
We emphasize the fact that while $\PP^{h,\star}$ solution of this problem is a measure that respects the initial condition and the price constraints similarly to the continuous problem, it is not obvious that it is a discrete diffusion Markov chain, as we used in the previous informal derivation. This is discussed in \cite{benamouEntropicMartingaleTransport}.

We will generalize this problem in the next section to allow for more general constraints on the marginals and the model prices, and we will then find its dual problem using Fenchel-Rockafellar duality.

%% file: duality.tex
\section{Duality}
We first recall the Fenchel-Rockafellar theorem and hence the form of the primal problem we aim to formulate.

\begin{theorem}[Fenchel-Rockafellar]
    \label{tfr}
    Let $(E, E^*)$ and $(F, F^*)$ be two couples of topologically paired spaces.
    Let $\Delta : E \rightarrow F$ be a continuous linear operator and $\Delta^\dagger : F^* \rightarrow E^*$ be its adjoint. Let $\Fc : E^* \to \overline{\R}$ and $\Gc : F^* \to \overline{\R}$ be two lower semicontinuous and proper convex functions. If there exists $\PP \in F^*$ such that $\Gc(\PP) < +\infty$ and $\Fc$ is continuous at $\Delta^\dagger \PP$, then :
    \begin{equation*}
        \sup_{\Phi \in E} -\Fc^\star(-\Phi) - \Gc^\star(\Delta \Phi) = \inf_{\PP \in F^*} \Fc(\Delta^\dagger \PP) + \Gc(\PP),
    \end{equation*}
    and the $\inf$ is attained. Moreover, if there exists a maximizer $\Phi^\star \in E$, then there exists $\PP^* \in F^*$ satisfying $\Delta \Phi^\star \in \partial \Gc(\PP)$ and $\Delta^\dagger \PP \in -\partial \Fc^\star(-\Phi^\star)$.
\end{theorem}

We note the primal problem :
\begin{equation}
    \mathcal{V} := \min_{\PP \in F^*} \Fc(\Delta^\dagger \PP) + \Gc(\PP) \label{frprim}\tag{PRIMAL},
\end{equation} 
and the dual :
\begin{equation}
    \mathcal{D} := \sup_{\Phi \in E} -\Fc^\star(-\Phi) - \Gc^\star(\Delta \Phi) \label{frdual}\tag{DUAL}.
\end{equation}
We will define their corresponding objects in the next sections.

\subsection{Primal problem}
In this part, we will formulate a generalization of the problem \eqref{mmmtprim}, in a standard
form compatible with the Fenchel-Rockafellar duality. We will first introduce the variables that we will control, and then define the primal problem.

We aim to define the objects corresponding to the primal problem, that is $\Fc$, $\Gc$ and the linear operator $\Delta^\dagger$.
The operator $\Delta^\dagger$ will include variables of interest that we will control. We will define these variables in the following remark.
\begin{definition}
    \label{rem:characteristics}
    For each timestep $k$, we define the following variables :
    \begin{align*}
        \nu_k(\diff x) & = X_k\#\PP^h(\diff x)                                              \\
        b_k(x)         & = \frac{1}{h}\expect_{\PP^h_{k,k+1}}*{B(X_k^h, X_{k+1}^h)|X_k^h=x} \\
        g_i            & = \expect_{\PP^h_{\tau_i}}*{G_i(X_{\tau_i})}.
    \end{align*}
    and We define $\Delta^\dagger$  as the following linear operator :
 
    \begin{equation*}
        \begin{array}{lrcl}
            \Delta^\dagger : & \ E^*:= \Pc(\Omega^h) & \to & F^*:= (\otimes_{k=1}^{N_T} \Pc(\Xc_k))  \times  (\otimes_{k=1}^{N_T} \Mc(\Xc_k)) \times
            \R^K                                                                                                              \\[8pt]
                             & \PP^h & \rightarrow & \Delta^\dagger\PP^h : = \left(\nu_k,\, \nu_k b_k, g_i\right)
        \end{array}
    \end{equation*}
\end{definition}
$\nu_k$, representing the marginal laws of $\PP^h$, are linear with respect to $\PP^h$, as they are a projection on a basis vector.
$g_i$ are the model prices at the calibration times and are also linear with respect to $\PP^h$, since they are the expectation of a function of the state at a given time.
$b_k$ are general variables corresponding to moments defined as the conditional expectation of a function $B$ between two consecutive timesteps as mentioned in section \ref{sec:prelim}. Since they are a conditional expectation, they are not linear with respect to $\PP^h$.
For this reason, we instead consider the product $\nu_k b_k$ that is linear with respect to $\PP^h$, as is done in the classical and martingale optimal transport \cite{BBA} \cite{Tan_2013}. Indeed, this product is equal to:
\[
    \nu_k b_k(x) = \frac{1}{h}\expect_{\PP^h_{k,k+1}(x, \cdot)}*{B(x, X_{k+1}^h)}
\]
which is linear with respect to $\PP^h$ as a simple expectation with respect to a projection.

Next, we define our function $\Gc$ as the scaled Kullback-Leibler divergence between $\PP^h$ and a reference measure $\bPP^h$, as it is the only function in our problem that depends on the measure itself.
\begin{definition}
    We define $\Gc$ the following functional :
    \begin{equation*}
        \Gc : \PP \in F^* \to h \KL(\PP|\bPP^h).
    \end{equation*}
\end{definition}

Finally, $\Fc$ will be a functional of the variables defined in remark \ref{rem:characteristics}, and will be a sum of convex functions of these variables, as presented for the continuous case at the end of Section \ref{sec:prelim}. These functions may be regularization, soft, or hard constraints.
\begin{remark}
    Let us say we want to impose a variable $x$ to be close to a target $x_0$.

    A \textbf{hard constraint} is of the form :
    \[
        F(x) = \begin{cases}
            0       & \text{if } x = x_0 \\
            +\infty & \text{otherwise}\end{cases}
    \]

    Some example of \textbf{soft constraints} are :
    \begin{itemize}
        \item $F(x) = C(x - x_0)^2$ : the quadratic penalty
        \item $F(x) = x\log(x/x_0) + x - x_0$ : the Kullback-Leibler divergence
    \end{itemize}
\end{remark}

We want to keep the formulation of our problem general and be able to express at least problem \eqref{mmmtprim} in this framework. We hence want to be able to express :
\begin{itemize}
    \item The constraint $X_0\#\PP^h = \delta_{x_0}$, but more generally we will consider any constraint on the marginals of $\PP^h$, leading to the definition of a term $\Fc_{\text{marg}} = \sum_{k = 0}^{N_T} M_k(\nu_k)$ which is a sum of convex, l.s.c., potentially null functions of the marginals.
    \item The constraint $\expect_{\PP^h}*{G_i(X_{\tau_i})} = g_i$, but more generally we will consider any constraint on the model prices, leading to defining a term $\Fc_{\text{prices}} = \sum_{i=1}^{N_C} C_i(g_i)$ which is a sum of convex, l.s.c., potentially null functions of the model prices.
    \item The constraints on moments of the process lead to the definition of a term $\Fc_{\text{mom}} = \sum_{k=0}^{N_T - 1} \Fc_{\text{mom}, k}(\nu_k, \nu_k b_k)$ that is a sum of perspective functions of the marginals and the product of the marginals and the moments.
\end{itemize}
which leads to the following definition of $\Fc$ :
\begin{definition}
    We define $\Fc$ the following functional :
    \begin{align*}
        \Fc(\nu, b, g) & = \Fc_{\text{marg}} + \Fc_{\text{mom}} + \Fc_{\text{prices}}                                                           \\
                       & = \sum_{k=0}^{N_T - 1} \Fc_{\text{mom}, k}(\nu_k, \nu_k b_k) + \sum_{k=0}^{N_T} M_k(\nu_k) + \sum_{i=1}^{N_C} C_i(g_i)
    \end{align*}
\end{definition}

\begin{remark}
    Problem \eqref{mmmtprim} may be retrieved from problem
    the problem we built
    with $\Fc_{\text{mom}, k}(\nu_k,  \nu_k b_k) = h\, \expect_{\PP^h}{F( { \nu_k b_k/h \nu_k})}, \, \forall k$, $M_0$ and $C_i, \, \forall i \in \Ic$ hard constraints and $M_k = 0,\, \forall k \neq 0$, and by taking $b_k$ to be the vector of variables $\beta_k$ and $\alpha_k$ as explained at the end of section \ref{sec:prelim}.
    This function $\Fc_{\text{mom}, k}$ is then a perspective function and is well studied in the theory of convex optimization.
\end{remark}

\subsection{Dual}

In order to apply the Fenchel-Rockafellar duality, we need to compute the Legendre transforms of $\Fc$ and $\Gc$, and the adjoint operator of $\Delta^\dagger$.

\begin{lemma}
    $\Gc^\star$, the Legendre-Fenchel transform of $\Gc$ is given by :
    \[
        \begin{array}{lrcl}
            \Gc^\star : & F = \C_b(\Omega^h)  & \to & \overline{\R}                                                \\[8pt]
                        & f & \rightarrow & h \E_{\bPP^h}(\exp(f/h))
        \end{array}
    \]
\end{lemma}
\begin{proposition}
    The Legendre transform of $\Fc$ is given by :
    \begin{align*}
        \Fc^\star(\phi_\nu, \phi_b, \lambda_g) = & \sum_k \inf_{\psi} M_k^\star(\phi_{\nu_k} - \psi) + F_k^\star(\psi, \phi_b)   \\
                                                 & + \sum_{i=1}^{N_C} C_i^\star(\lambda_{g_i}) + M_{N_T}^\star(\phi_{\nu_{N_T}})
        \label{fdual}
        \tag{$\Fc$-dual}
    \end{align*}
\end{proposition}
\begin{proof}
    We first notice that the expression of $\Fc$ is separable in $k$ and $i$ and that $g_i$ can be separated from $(\nu_k,  \nu_k b_k)$. We can then rewrite the function as :

    \begin{equation}
        \Fc(\nu, \nu b, g) = \sum_{k = 0}^{N_T - 1} \Fc^{1}_k(\nu_k,  \nu_k b_k) + \sum_{i = 0}^{N_C} \Fc^{2}_i(g_i) + M_{N_T}(\nu_{N_T})
    \end{equation}
    where :
    \begin{align*}
        \Fc_k^1(\nu_k,  \nu_k b_k) & = \Fc_{\text{mom},k}(\nu_k,  \nu_k b_k) + M_k(\nu_k) \\
        \Fc_i^2(g_i)               & = C_i(g_i).
    \end{align*}
    The Legendre transform of $\Fc_i^2$ are the Legendre transform of $C_i$ and the Legendre transform of $\Fc_k^1$ is given by :
    \begin{align*}
        \Fc_k^{1 \star}(\phi_{\nu}, \phi_{b}) & = \sup_{\nu, b} \, \phi_{\nu} \nu + \phi_{b} b - \Fc_k^1(\nu, b) - M_k(\nu)                                          \\
                                              & = \sup_{\nu, \nu', b} \inf_\psi \psi(\nu' - \nu) + \phi_{\nu} \nu + \phi_{b} b - \Fc_k^1(\nu', b) - M_k(\nu)         \\
                                              & = \inf_\psi \sup_{\nu, \nu', b} (\phi_{\nu} - \psi) \nu + \phi_{b} b + \psi \nu' - \Fc_k^1(\nu', b) - M_k(\nu)       \\
                                              & = \inf_\psi \sup_{\nu} (\phi_{\nu} - \psi) \nu - M_k(\nu) + \sup_{\nu', b} \psi \nu' + \phi_{b} b - \Fc_k^1(\nu', b) \\
                                              & = \inf_\psi M_k^\star(\phi_{\nu} - \psi) + F_k^\star(\psi, \phi_{b})
    \end{align*}

    We can conclude by summing over $k$ and $i$ the separable parts.
\end{proof}

We are particularly interested in the case $\Fc_k(\nu_k, \nu_k b_k) = h \, \E_{\nu_k}(F(\nu_k b_k/h\nu_k))$ where $F$ is a convex function, so $\Fc_k$ is a perspective function. This is because it is a direct discretisation of the continuous problem. For this case, we first find the Legendre-Transform of $\Fc_k$ coordinate-wise~:
\begin{align*}
    \Fc_k^\star(\phi_\nu, \phi_b) & = \sup_{\nu, b} \, \phi_\nu \nu + \phi_b b - \Fc_k(\nu, b)                             \\
                                  & = \sup_{\nu, b} \, \phi_\nu \nu + \phi_b b - h\nu F(b/h\nu)                            \\
                                  & = \sup_{\nu} \, \phi_\nu \nu + h\nu \left( \sup_{b} \phi_b (b/h\nu) - F(b/h\nu)\right) \\
                                  & = \sup_{\nu} \, \phi_\nu \nu + h\nu F^\star(\phi_b)                                    \\
                                  & = \sup_{\nu} \, \left(\phi_\nu + hF^\star(\phi_b)\right) \nu                           \\
                                  & = \begin{cases}
                                          0       & \text{if } \phi_\nu + hF^\star(\phi_b) < 0 \\
                                          +\infty & \text{otherwise}
                                      \end{cases}
\end{align*}

We can hence obtain the Legendre-Transform of $\Fc_k^1$ defined above as~:
\begin{align*}
    \Fc_k^1(\phi_\nu, \phi_b) & = \inf_\psi M_k^\star(\phi_{\nu} - \psi) + \Fc_k^\star(\psi, \phi_{b_k}) \\
                              & = M_k^\star(\phi_{\nu} + hF^\star(\phi_b))
\end{align*}

The dual of $\Fc$ is given by the following simpler expression~:
\begin{align*}
    \Fc^\star(\phi_\nu, \phi_b, \lambda_g) = & \sum_k M_k^\star(\phi_{\nu_k} + hF^\star(\phi_b) )                            \\
                                             & + \sum_{i=1}^{N_C} C_i^\star(\lambda_{g_i}) + M_{N_T}^\star(\phi_{\nu_{N_T}})
\end{align*}

\begin{lemma}
    The adjoint operator to $\Delta^\dagger$ is given by :
    \begin{align*}
        \Delta : \Phi = \left(\phi_{\nu_k},\, \phi_{b_k},\, \lambda_{g_i}\right) \to & \bigoplus_k h \phi_{\nu_k} + B(x_{k}, x_{k+1}) \phi_{b_k} \\
                                                                                         & + \sum_i G_i(x_{\tau_i}) \lambda_{g_i}
    \end{align*}
\end{lemma}

\begin{proof}
    We can easily check that $\SP{\Delta \Phi}{\PP^h}{} = \SP{\Phi}{\Delta^\dagger \PP^h}{}$.
\end{proof}

As customary  in optimal transport, through the Monge-Kantorovitch dual formulation , optimal solutions (but for non-optimal constraints) can be obtained from the dual potentials. Thus, by constructing a maximizing sequence of the dual problem, we hope to find a  sequence of measures converging to a solution of the primal problem. This is the approach we follow in this work.
\begin{proposition}
    Let $\Phi^\star = (\phi_\nu^\star, \phi_b^\star, \lambda_g^\star)$  be a solution of \ref{frdual}. Then $\Phi^\star$ induces a measure $\PP^{h,\star}$ through
    \label{prop:optimaldensity}
    \begin{equation*}
        \frac{\diff \PP^{h,\star}}{\diff \bPP^h } = \exp\left(\frac{\Delta \Phi^\star}{h}\right),
    \end{equation*}
    $\PP^{h,\star}$ is the optimal solution of (\ref{frprim})
    for the constraints functions $M,C$ that are finite under $\PP^{h,\star}$.
\end{proposition}

\begin{proof}
    The first optimality condition is given by :
    \begin{equation*}
        \Delta \Phi^\star \in h \, \partial_{\PP^{h, \star}} \KL(\PP^{h,\star}|\bPP^h)
    \end{equation*}
    which leads to :
    \begin{equation}
        \Delta \Phi^\star = h \, \log\left(\frac{\diff \PP^{h,\star}}{
            \diff \bPP^h}\right)
    \end{equation}
\end{proof}

\begin{definition}
    Let $\Delta_{i, i+1}$ the transitional part of $\Delta$ given by:
    \begin{equation*}
        \Delta_{i, i+1}(x_i, x_{i+1}) = B(x_{i}, x_{i+1}) \phi_{b_i}(x_i),
    \end{equation*}
    Let $\psi^u_i, \psi^d_i$ potential functions computing the forward and backward influence :
    \begin{align*}
        \psi^u_i(x_i) & = \log\left(\int \rho_0 \prod_{k=0}^{i-1} \exp\left(\frac{\Delta_{k,k+1}}{h} + \phi_{\nu_k} + \vv{\Lambda_k} \cdot \vv{G_k} \right) \bPP_{k,k+1}^h \diff x_{[0, i-1]}\right)         \\
        \psi^d_i(x_i) & = \log\left(\int \prod_{k=i}^{N_T-1} \exp\left(\frac{\Delta_{k,k+1}}{h} + \phi_{\nu_{k+1}} + \vv{\Lambda_{k+1}} \cdot \vv{G_{k+1}}\right) \bPP_{k,k+1}^h \diff x_{[i+1, N_T]}\right)
    \end{align*}
    and let $\vv{\Lambda}_k$ the vector of Lagrange multipliers and $\vv{G_k}$ the vector of their corresponding payoff functions associated with timestep $k$:
    \begin{align*}
        \vv{\Lambda}_k & = \left(\lambda_{g_i}\right)_{i \in \Ic_k} \\
        \vv{G_k}       & = \left(G_i\right)_{i \in \Ic_k}
    \end{align*}
\end{definition}

One interesting property of this formulation is that the Markovianity of the solution directly arises from the structure of the linear operator $\Delta$.
\begin{proposition}
    \label{prop_markovian}
    Let $\PP^{h,\star}$ an optimal solution of \eqref{frprim}. The following properties hold true :
    \begin{itemize}
        \item Its joint density with respect to the reference measure is given by :
              \begin{align*}
                  \diff \PP^{h,\star}_{i, i+1}(x_i, x_{i+1}) = & \exp(\psi^u_i(x_i) + \phi_{\nu_{i}}(x_i) + \vv{\Lambda_i} \cdot \vv{G_i}(x_i)                                                            \\
                                                               & \hphantom{\exp(} + \Delta_{i, i+1}(x_i, x_{i+1})/h                                                                                       \\
                                                               & \hphantom{\exp(} + \vv{\Lambda_{i+1}} \cdot \vv{G_{i+1}}(x_i) + \phi_{\nu_{i+1}}(x_{i+1}) + \psi^d_{i+1}(x_{i+1})) \diff \bPP_{i, i+1}^h
              \end{align*}
              From the form of its transition density, and the Markovianity of the reference measure, we deduce that it is Markovian.
        \item Its marginal are given by :
              \begin{align*}
                  \nu_k^\star = \PP^{h,\star}_k(x_k) & = \exp(\psi^u_k(x_k) + \phi_{\nu_k}(x_k) + \vv{\Lambda_k} \cdot \vv{G_k}(x_k) + \psi^d_k(x_k))
              \end{align*}
    \end{itemize}

\end{proposition}
\begin{proof}
    In Appendix \ref{proof_propmarkov}
\end{proof}

\begin{proposition}
    The quantities $\psi^u_k$, $\psi^d_k$ can be computed iteratively using the following updates :
    \begin{align*}
        \psi^{u}_{k+1} & = \log\left(\int \exp\left(\psi^{u}_k + \frac{\Delta_{k,k+1}}{h} + \phi^{n}_{\nu_k} + \vv{\Lambda^n_k} \cdot \vv{G_k} \right) \bPP_{k,k+1}^h \diff x_k\right)       \\
        \psi^{d}_{k-1} & = \log\left(\int \exp\left(\psi^{d}_k + \frac{\Delta_{k-1,k}}{h} + \phi^{n}_{m_{k}} + \vv{\Lambda^n_{k}} \cdot \vv{G_{k}} \right) \bPP_{k-1,k}^h \diff x_{k}\right)
    \end{align*}
    with $\psi^u_0 = \log \rho_0$ and $\psi^d_{N_T} = 0$.
\end{proposition}
\begin{proof}
    Let $i \in \Kc$ and $x_i \in \Xc_i$. We can compute the following integral :
    \begin{align*}
        \psi^u_i(x_i) & = \log\left(\int \rho_0 \prod_{k=0}^{i-1} \exp\left(\frac{\Delta_{k,k+1}}{h} + \phi_{\nu_k} + \Lambda_k \cdot \vv{G_k} \right) \bPP_{k,k+1}^h \diff x_{[0, i-1]}\right)                    \\
                      & = \log\left(\int \left(\int \rho_0 \prod_{k=0}^{i-2} \exp\left(\frac{\Delta_{k,k+1}}{h} + \phi_{\nu_k} + \Lambda_k \cdot \vv{G_k} \right) \bPP_{k,k+1}^h \diff x_{[0, i-2]}\right) \right. \\
                      & \hphantom{= \log\left(\int \right.} \left. \exp\left(\frac{\Delta_{i-1,i}}{h} + \phi_{m_{i-1}} + \Lambda_{i-1} \cdot \vv{G_{i-1}} \right) \bPP_{i-1,i}^h \diff x_{i-1}\right)              \\
                      & = \log\left(\int \exp\left(\psi^{u}_{i-1} + \frac{\Delta_{i-1,i}}{h} + \phi_{m_{i-1}} + \Lambda_{i-1} \cdot \vv{G_{i-1}} \right)\bPP_{i-1,i}^h \diff x_{i-1}\right).
    \end{align*}
    A symmetric reasoning can be done for $\psi^d_i$.
\end{proof}

\section{Sinkhorn algorithm}
In order to numerically solve the dual problem, we propose to use an extension of the Sinkhorn algorithm. The Sinkhorn algorithm is a well-known algorithm to solve optimal transport problems. In order to do such numeric computations, we need to discretise our problem in space as well.

\subsection{Algorithm}
We first describe the principle of the algorithm.

We initialise the dual potentials $\phi_{\nu_k}$, $\phi_{b_k}$ and $\vv{\Lambda_k}$ as $0$ which corresponds to being equal to the reference neasure, and we then iteratively the following updates :
\begin{equation}
    \tag{SK1}
    \label{sinkhorn}
    \begin{cases}
        \phi_{\nu_k}^{n+1} = \arg\,\min_{\phi} \inf_{\psi} -M_k^\star(-\phi - \psi) - F_k^\star(\psi, -\phi_b)                                                                                              \\
        \hphantom{\phi_{\nu_k}^{n+1} = \arg\,\min_{\phi} } - h\expect_{\bPP_k^h}*{\exp(\psi^{u,n+1}_k + \phi + \vv{\Lambda^n_k} \cdot \vv{G_k} + \psi^{d,n}_k)}                                             \\
        \vv{\Lambda_k^{n+1}} = \arg\,\min_{\vv{\Lambda}} -C_k^\star(-\Lambda) - h\expect_{\bPP_k^h}*{\exp(\psi^{u,n+1}_k + \phi^{n+1}_{\nu_k} + \vv{\Lambda} \cdot \vv{G_k} + \psi^{d,n}_k)}                \\
        \phi_{b_k}^{n+1} = \arg\,\min_{\phi} \inf_{\psi} -M_k^\star(-\phi_{\nu_k} - \psi) - F_k^\star(\psi, -\phi)                                                                                          \\
        \hphantom{\phi_{p_k}^{n+1} = \arg\,\min_{\phi}} - h\expect_{\bPP_{k,k+1}^h}*{\exp(\psi^{u,n+1}_k + \phi_{\nu_k}^{n+1} + \vv{\Lambda^{n+1}_k} \cdot \vv{G_k} + \Delta_{i,i+1}(\phi) + \psi^{d,n}_k)} \\
    \end{cases}
\end{equation}
and for the last marginal :
\begin{equation*}
    \tag{SK2}
    \label{sinkhorn2}
    \begin{cases}
        \phi_{m_{N_T}}^{n+1} = \arg\,\min_{\phi} -M_{N_T}^\star(-h\phi)                                                                                                      \\
        \hphantom{\phi_{m_{N_T}}^{n+1} = \arg\,\min_{\phi} } - h\expect_{\bPP_{N_T}^h}*{\exp(\psi^{u,n+1}_{N_T} + \phi + \Lambda^n_{N_T} \cdot \vv{G_{N_T}})}                \\
        \Lambda_i^{n+1} = \arg\,\min_{\Lambda} -C_i^\star(-\Lambda) - h\expect_{\bPP_{N_T}^h}*{\exp(\psi^{u,n+1}_{N_T} + \phi^{n+1}_{m_{N_T}} + \Lambda \cdot \vv{G_{N_T}})} \\
    \end{cases}
\end{equation*}

In practice, before each step we compute every downward $\psi^d_k$, iteratively as described in the previous section.
The pseudocode algorithm is described in Algorithm \ref{algo:psiuppsidown}.

\begin{algorithm}[H]
    \label{algo:psiuppsidown}
    \caption{Function Definitions for UpdatePsiUp and UpdatePsiDown}
    \SetAlgoNlRelativeSize{0}
    \SetAlgoNlRelativeSize{-2}
    \SetNlSty{textbf}{(}{)}
    \SetAlgoNlRelativeSize{-2}
    \SetNlSty{}{}{}
    \SetAlgoNlRelativeSize{-4}
    \SetNlSty{}{}{}

    \textbf{Function} UpdatePsiUp(\( \psi^{u}_{k+1}, \phi^{n}_{\nu_k}, \Lambda^n_k, \bPP_{k, k+1}^h, x_k, x_{k+1}, \phi_{b_k}, h \)) \\
    \Begin{
        \textbf{Compute} \( \Delta_{k, k+1} = (x_{k+1} - x_k) \phi_{b_k}(x_k) \) \\
        \Return \( \log\left(\int \exp\left(\psi^{u}_{k+1} + \frac{\Delta_{k,k+1}}{h} + \phi^{n}_{\nu_k} + \Lambda^n_k \cdot \vv{G_k} \right) \bPP^h_{k,k+1} \diff x_{k-1}\right) \)
    }

    \textbf{Function} UpdatePsiDown(\( \psi^{d}_{k-1}, \phi^{n}_{\nu_k}, \Lambda^n_k, \bPP^h_{k-1, k}, x_k, x_{k-1}, \phi_{b_k}, h \)) \\
    \Begin{
        \textbf{Compute} \( \Delta_{k-1, k} = (x_k - x_{k-1}) \phi_{b_k}(x_k) \) \\
        \Return \( \log\left(\int \exp\left(\psi^{d}_{k-1} + \frac{\Delta_{k-1,k}}{h} + \phi^{n}_{m_{k}} + \Lambda^n_{k} \cdot \vv{G_{k}} \right) \bPP^h_{k-1,k} \diff x_{k}\right) \)
    }
\end{algorithm}

We describe the pseudo-code of the Multi-Marginal Sinkhorn algorithm in Algorithm \ref{algo:sinkhorn}.

\begin{algorithm}[H]
    \caption{Sinkhorn algorithm for problem \ref{frdual}}
    \label{algo:sinkhorn}
    \KwIn{Number of timesteps $N_T$, support of each space $\Xc_k^{\diff x}$}
    \KwIn{Stopping tolerance $\epsilon$, reference measure $\bPP^h$}
    \KwIn{Initial potentials $\phi_{\nu_k}^0$, $\phi_{b_k}^0$, $\Lambda_i^0$}
    \KwResult{Numerical solution of problem \ref{frdual}}
    $\psi^{u,0}_0 \gets \log \rho_0$\;
    $\psi^{d,0}_{N_T} \gets 0$\;
    \For{$n \gets 0$ \KwTo $N$}{
        \For{$k \gets N_T-1$ \KwTo $0$}{
            $\psi^{d,n}_k \gets$ UpdatePsiDown($\psi^{d,n}_k$, $\phi_{\nu_k}^{n}$, $\phi_{b_k}^{n}$, $\Lambda^n_i$, $\bPP^h$)\;
        }
        \For{$k \gets 0$ \KwTo $N_T - 1$}{
            $\phi_{\nu_k}^{n+1} \gets$ SolveMarginal($\phi_{\nu_k}^{n}$, $\phi_{b_k}^{n}$, $\Lambda^n_i$, $\psi^{u,n}_k$, $\psi^{d,n}_k$, $\bPP^h$)\;
            $\Lambda_k^{n+1} \gets$ SolvePrices($\phi_{\nu_k}^{n}$, $\phi_{b_k}^{n}$, $\Lambda^n_i$, $\psi^{u,n}_k$, $\psi^{d,n}_k$, $\bPP^h$)\;
            $\phi_{b_k}^{n+1} \gets$ SolveDriftVol($\phi_{\nu_k}^{n+1}$, $\phi_{b_k}^{n}$, $\Lambda^{n+1}_i$, $\psi^{u,n}_k$, $\psi^{d,n}_k$, $\bPP^h$)\;
            $\psi^{u,n+1}_{k+1} \gets$ UpdatePsiUp($\psi^{u,n+1}_{k+1}$, $\phi_{\nu_k}^{n+1}$, $\phi_{b_k}^{n+1}$, $\Lambda^{n+1}_i$, $\bPP^h$)\;
        }
        $\phi_{m_{N_T}}^{n+1} \gets$ SolveMarginal($\phi_{m_{N_T}}^{n}$, $\Lambda^{n}_i$, $\psi^{u,n}_{N_T}$, $\psi^{d,n}_{N_T}$, $\bPP^h$)\;
        $\Lambda_{N_T}^{n+1} \gets$ SolvePrices($\phi_{m_{N_T}}^{n+1}$, $\Lambda^{n}_i$, $\psi^{u,n}_{N_T}$, $\psi^{d,n}_{N_T}$, $\bPP^h$)\;
        $e_{\text{max}} \gets \frac{\|\Phi^{n+1} - \Phi^n\|_\infty}{\|\Phi^n\|_\infty}$\;
        \If{$e_{\text{max}} < \epsilon$}{
            \Return $\Phi$, $\Psi$\;
        }
    }
    \Return $\Phi$, $\Psi$\;
\end{algorithm}
The functions SolveMarginal, SolvePrices and SolveDriftVol are functions that solve the minimization problems in \eqref{sinkhorn} and \eqref{sinkhorn2},
and might have different implementations depending on the structure of the problem.

In order to provide a numerical implementation of the method, we provide multiple details in the next sections.

\subsection{Truncation in space}
\label{truncate}
First, we need to truncate the support so that it has a finite width. We take advantage of the fact that our reference measure $\PP^h$ is a diffusion measure and that its density tends quickly to $0$ as we move away from the mean.
We want to create an interval for each marginal $\nu_k$ that contains most of the mass.
Because the marginals are not known beforehand, we instead propose to truncate on $\overline{\nu_k}$, the marginals of the reference measure $\bPP^h$.

First, in the case of a reference measure with constant drift and volatility, for each timestep $t_k =  k\, h \quad i = 0.,\ldots,N_T$,  we restrict the computational domain to
\[
    \overline{\Xc_k} = \left[m_k - \delta \,  v_k ,\, m_k +  \delta \, v_k  \right]
\]
where :
\begin{itemize}
    \item[-] $m_k = m_0 + h \, \sum_{l = 0}^{k} \overline{\mu_l}$ is the mean of the $k$-th marginal, where $m_0$ is the initial mean of the reference measure.
    \item[-] $\delta v_i$ is a multiple of the standard deviation of the reference measure: $v_k = \sqrt{v_0^2 + h \sum_{l = 0}^{k} \overline{\sigma_l}^2}$ where $v_0$ is the standard deviation of the initial marginal of the reference measure. For a sufficiently large $\delta$ (in general $\delta =5$)
        we expect $\PP^h_k$ the solution to be negligible outside of $\Xc_k$, i.e.
        the mass transported by the drift (small as we are solving
        soft martingale constraint problems) and the  diffusion further than the
        enlarged domain is negligible.
\end{itemize}

When the reference measure has a non-constant drift and volatility, we can use the same truncation as above, but with the maximum in space of the drift and volatility.

\subsection{Discretisation}

On this compact supports, we discretize the potentials on a regular grid.
We can hence represent the potentials $\phi_{\nu_k}$, $\phi_{b_k}$, as vectors $\phi_{b_k}^{\diff x}$, of size $N_{\Xc_k}$ and $N_{\Xc_k}\times d$ respectively. We can also represent the quantities $\psi^u_k$, $\psi^d_k$ as vectors $\psi^{u,\diff x}_k$, $\psi^{d,\diff x}_k$ of size $N_{\Xc_k}$.
The integrals can be replaced by discrete sums over the grid $\Xc_k^{\diff x}$.
For clarity, we will drop the superscript $\diff x$ in the following and keep the notations with integrals as they can be used interchangeably.

An important parameter is the space discretisation step $\diff x$. We want to choose it
so that the transition law $\bPP_{i,i+1}$ is non-zero on enough points.
To ensure this, we want :
\[
    \diff x < K \bs_i \sqrt{h}
\]
which links $N_{\Xc_k}$ to $N_T$, for a constant $K$ to be determined which ensure the minimum of points in the width of the kernel.

\subsection{Multiscale Strategy}
\label{sec:multiscale}
Because the complexity of algorithm \ref{algo:sinkhorn} scales with
the number of timesteps $N_T$, it can be interesting to start with a coarse time discretisation and to refine it iteratively.

One way to do so is presented in \cite{benamouEntropicMartingaleTransport} and consists in interpolating
the potentials $\phi_{b_k}$ and using them as initialisation for the next level of discretisation.

Here, we propose an alternative method, which consists in using the result of optimization at a coarse scale as the reference measure for a finer scale. This is done as such :
\begin{enumerate}
    \item Solve the problem at a scale $h = \frac{T}{N_T - 1}$ with $N_T$ timesteps. \label{step1}
    \item Interpolate the obtained volatilities $\sigma_k^2(x, t)$ for a new scale $h' = \frac{T}{N_T' - 1}$ with $N_T' > N_T$ timesteps. Multiple interpolation technique can be used, in our case, we use Unbalanced Optimal Transport Barycenters. We then obtain a new reference measure $\bPP^{h'}$.
    \item $h = h'$, $N_T = N_T'$, and return at step \ref{step1} if the desired scale is not reached.
\end{enumerate}

We can, for instance, start with $N_T$ equal to the number of calibration steps.

\subsection{Anderson Acceleration}
Entropic Sinkhorn iterations are known to converge slowly when the regularization parameter $\epsilon \rightarrow 0$.
Since one Sinkhorn iteration $\Phi^{k+1} = s(\Phi^k)$ is a fixed-point iteration, in order to accelerate convergence, we propose to use Anderson acceleration \cite{anderson1965} \cite{anderson2011}.
We consider our variable $\Phi$ as a variable of dimension $N_\Phi$ and denote $g(\Phi) = s(\Phi) - \Phi$ is the residual of the Sinkhorn iteration.
In particular, given a vector of $m_k$ iterates $\Phi^k_{m_k} = (\Phi^{k-m_k+1}, \dots, \Phi^{k})$, we denote $G = \{g(\Phi^k)\}_k = \{s(\Phi^k) - \Phi^k\}_k$ the vector of their residuals.
We compute the next iterate $\Phi^{k+1}$ as a combination of the $m_k$ previous iterates as a solution to the following problem :
\begin{equation}
    \min_{\alpha \in \Delta(m_k)} ||G\alpha||_2
    \tag{AA}
    \label{eq:anderson}
\end{equation}
where $\Delta(m_k) = \{v \in \R^{m_k}, \sum v_i = 1\}$. The idea is that by first-order approximation, $||g(\Phi^k_{m_k} \alpha)||_2 \approx ||G \alpha||_2$.
The new iterate is then given by $\Phi^{k+1} = s(\Phi^k_{m_k})\alpha \approx s(\Phi^k_{m_k}\alpha)$ by first order approximation.

The problem (\ref{eq:anderson}) is reformulated with the following linear least square problem as suggested in \cite{anderson1965} \cite{anderson2011}:
\begin{equation}
    \min_{\gamma \in \R^{m_k-1}} ||\Gc\gamma - g_k||_2
    \tag{AA-LS}
    \label{eq:anderson_ls}
\end{equation}
where $\Gc = \begin{pmatrix} g_k - g_{k-1} & \dots & g_{k-m_k+1} - g_{k-m_k} \end{pmatrix} \in \R^{N_\Phi \times (m_k-1)}$ with $g_k = g({\Phi^k}) \in \R^{N_{\Phi}}$.
This provides a solution to (\ref{eq:anderson}) as : $\alpha = \begin{pmatrix} 1 - \gamma_{m_k-1} & \gamma_{m_k-1} - \gamma_{m_k-2} & \dots & \gamma_2 - \gamma_1 & \gamma_1 \end{pmatrix}$.

We can then rewrite the iterate $\Phi^{k+1}$ as :
\begin{equation*}
    \Phi^{k+1} = s(\Phi^k_{m_k})\alpha = \sum_{i=k-m_k}^{k} \alpha_i s(\Phi^{i}) = s(\Phi^k) + \sum_{i=k-m_k}^{k-1} \gamma_i (s(\Phi^{i+1}) - s(\Phi^i))
\end{equation*}
Noting $d\Phi^k = \begin{pmatrix} \Phi^k - \Phi^{k-1} & \dots & \Phi^{k-m_k+1} - \Phi^{k-m_k} \end{pmatrix} \in \R^{N_\Phi \times (m_k-1)}$ the matrix of iterate differences, we can rewrite the iterate $\Phi^{k+1}$ as :
\begin{align}
    \Phi^{k+1} & = s(\Phi^k) + d\Phi^k \gamma \nonumber                                   \\
               & = \Phi^k + g_k - (\Gc + d\Phi^k) \gamma \tag{AA-iter} \label{eq:aa-iter}
\end{align}

The problem (\ref{eq:anderson_ls}) is a linear least square problem that can be solved efficiently. We solve it by solving the linear system
$\Gc^T \Gc \gamma = \Gc^T g_k$ since the Gram matrix $\Gc^T \Gc$ can be easily computed.
To avoid cases where $\Gc^T \Gc$ is ill-conditioned, we add a ridge regularization term $\epsilon I$ to the Gram matrix, where $\epsilon$ is a small positive constant. When $\epsilon \rightarrow 0$,
the solution of the linear system converges to the solution of (\ref{eq:anderson_ls}), while when $\epsilon \rightarrow \infty$, the solution of the linear system converges to $0$ and hence the steps become simple Sinkhorn iterations.

To avoid convergence problems, we also employ a safeguarding mechanism as described, for instance, in \cite{garstka2022safeguarded} to enhance the stability and performance of the Anderson acceleration technique. At each iteration \( k \), an accelerated candidate \( \Phi^{k, \text{acc}} \) is generated using iteration (\ref{eq:aa-iter}). To assess the quality of this candidate, we compare its residual \( g(\Phi^{k, \text{acc}}) \) with the residual of the last accepted iterate \( g(\Phi^{k}) \). Specifically, we impose a safeguarding condition:
\begin{equation}
    \|  g(\Phi^{k+1, \text{acc}}) \|_2 \leq \tau \| g(\Phi^{k}) \|_2
\end{equation}
Here, \( \tau \) is a tolerance parameter. If the condition is met, the accelerated candidate \( \Phi^{k, \text{acc}} \) is accepted for the next iteration. Otherwise, the candidate is declined, and the algorithm proceeds without acceleration for that step.

Finally, similarly to what is done for Sinkhorn algorithm, we stop the algorithm when the $l^\infty$ norm of the residual is below a given tolerance $\epsilon_{\text{stop}}$.

%% file: applications.tex
\section{Application: Local volatility calibration}
We want to minimize over martingale positive process, hence we are interested
in the process of the form $X_t = \log S_t$ where $S_t$ is a martingale diffusion
process:
\begin{equation*}
    \frac{\diff S_t}{S_t} = \sigma(S_t, t), {\diff \RR}_t.
\end{equation*}
Applying Ito's lemma, in terms of $X_t = \log S_t$, we obtain the following SDE:
\begin{equation*}
    \diff X_t = -\frac{1}{2} \sigma^2(X_t, t) \, \diff t + \sigma(X_t, t) \, {\diff \RR}_t.
\end{equation*}
As noted in article \cite{guo2021optimal}, in terms of the characteristics of the SDE $\mu$ and $\sigma^2$, this reads:
\[
    2 \mu (x, kh) = -\sigma^2(x, kh),
\]
Using our defined variables $\beta_k$ and $\alpha_k$ in Section \ref{sec:prelim} equations (\ref{beta_def}) (\ref{alpha_def}), this can be written as
\[
    2 \beta_k = -\alpha_k.
\]
Notice however that both equations are not exactly equivalent due to the first-order term $\mu^2 \, h$ appearing in equation $\eqref{alpha_def}$.

Instead of the variables $\beta$ and $\alpha$, we will use the variable $b_k$ corresponding to the choice $B(X, Y) = 1 - e^{Y - X}$. We compute $b_k$ :
\begin{align*}
    b_k(x) & = \frac{1}{h} \expect*{1 - e^{X_{k+1} - X_k}|X_k=x}                    \\
           & = \frac{1}{h} \frac{1}{e^{X_k}}\expect*{e^{X_k} - e^{X_{k+1}}|X_k = x} \\
           & = \frac{1}{h} \frac{1}{S_k}\expect*{S_k - S_{k+1}|X_k = x},
\end{align*}
which directly corresponds to an at  martingale constraint on the process $S_t$. This is the moment we will consider in this application and the above strategy method is easily applied. \\

We penalize this variable with $\Fc$ a soft penalization $F = c_{\text{mart}} \| \cdot \|_{L^2}$ with a constant penalization parameter $c_\text{mart}$.
This is to overcome the convex ordering problem mentioned in \cite{alfonsiSamplingProbabilityMeasures2017} when working on a discretized grid:
the discrete approximation on a grid of continuous measures in convex order might not be in convex order.

For the price constraints, let $c_i \in \R^+$ be an observed price, we use the soft constraint $C_i$ a convex function with minima in $c_i$, for instance, $C_i = \frac{1}{2}(\cdot - c_i)^2$.
We use the payoff function $g_i(x) = \max(0, e^x - K_i)$ for a call option with strike $K_i$, and $g_i(x) = \max(0, K_i - e^x)$ for a put option with strike $K_i$.

For the first marginal constraint, we propose using a hard constraint $M_0 = \iota_{\mu_0}$ whose dual is $\SP{\phi_{\nu_0}}{\mu_0}{}$, with $\mu_0 = \delta_{\log S_0}$. \\

Finally, we obtain the following problem:
\begin{align*}
    \mathcal{V} = \inf_{\PP^h} & \sum h \expect_{\nu_k}*{F(b_k)} + h \KL(\PP^h|\bPP^h)                 \\
                               & + \iota_{\mu_0}\left(\nu_0\right) + \sum_{i=1}^n C_i\left(g_i\right),
\end{align*}
and in its dual form :
\begin{align*}
    \mathcal{D} = \sup_{\phi_{\nu_0}, \phi_{b_k}, \lambda_{g_i}} J(\phi) \eqdef & \expect_{\mu_0}*{\phi_{\nu_0}} + \sum_{i=0}^{N_C} C_i^\star(\lambda_{g_i})           \\
                                                                                & + h \expect_{\bPP^h}*{\exp\left(\frac{\Delta(\phi_m, \phi_b, \lambda_g)}{h}\right)},
\end{align*}
under the constraints $\phi_{\nu_k} = hH^\star(-\phi_{b_k}), \, \forall k \notin \{0, N_T\}$, $\phi_{m_{N_T}} = 0$.

The first equation of the Sinkhorn system (\ref{sinkhorn}) becomes :
\begin{align*}
    \mu_0 & = \int \PP^h(x_0, \diff x_{-0})                              \\
          & = \exp(\psi^u_0(x_0) + \phi_{\nu_0}(x_0)/h + \psi^d_0(x_0)),
\end{align*}
which can be solved using classical Sinkhorn iteration as :
\begin{equation*}
    \phi_{\nu_0} = -h (\psi^u_0 + \psi^d_0 - \log(\mu_0)).
\end{equation*}

We derivate our functional $J$ with respect to $\phi_{b_k}$ to obtain the following optimality equation :
\begin{equation*}
    0 = \expect_{\PP^h_{k,k+1}}*{(1 - e^{X_{k+1} - X_{k}}) + hH^{\star'}(-\phi_{\nu_k}(x_k))}
\end{equation*}
The expression of the Hessian is simple as well, as it is diagonal.

In a similar fashion, we can obtain optimality equations for $\Lambda_i$ :
\begin{equation*}
    0 = -C_i^{\star'}(-\lambda_i) + \E_{\PP_k^h}[g^{\PP^h}_i], \, \forall i \in \{0, N_T-1\}
\end{equation*}
We can compute the Hessian which is this time of size $|Ic_k|$.

To optimize on these two variables, we can perform a Newton method.

The calculation of the derivatives on $\phi_{b_k}$ is of complexity $O(N_{\Xc_k} \times N_{\Xc_{k+1}})$, while the other
derivatives are of complexity $O(N_{\Xc_k})$. The computation of $\psi^u$ and $\psi^d$ are of complexity $O(N_{\Xc_k} \times N_{\Xc_{k+1}})$.
Hence, the complexity of the whole algorithm is of complexity $O(N_T \times \max N_{\Xc_k}^2)$.

At the coarsest scale, the reference measure is chosen to match the ATM price at the calibration time.
Hence, given the ATM volatility $\sigma_{\text{BS}}(F, \tau_i)$ for all calibration time $(\tau_i)$, where $F$ is the forward price,
we set our reference volatility to be $\overline{\sigma}^2_0(x) = \sigma_{\text{BS}}^2(F, \tau_0) \tau_0$ and
$\overline{\sigma}^2_{i}(x) = \sigma_{\text{BS}}^2(F, \tau_i) \tau_i - \overline{\sigma}^2_{i-1}(x)$ for all $i \in \{1, \dots, \}$.

\subsection{Numerical results}

Here, we first generate prices using a parametric local volatility surface.
The local volatility surface that we choose is the SSVI surface as presented in \cite{Gatheral_Jacquier_2012}.
We choose the at-the-money implied total variance for the money to be $\theta_t = 0.04t$. We choose a power-law parameterization of the function $\phi$
described in \cite{Gatheral_Jacquier_2012} as $\phi(\theta) = \eta \theta^{-\lambda}$. The at-the-money total implied variance is then
\[
    \sigma_{\text{BS}}^2(k, T) = \frac{\theta_t}{2} \left(1 + \rho \phi(\theta_t) k + \sqrt{(\phi(\theta_t) k + \rho)^2 + (1 - \rho^2)}\right),
\]
where $k$ is the log-moneyness $\log(K/F)$. The parameters are chosen as $\eta = 1.6$, $\lambda = 0.4$ and $\rho = -0.15$.
The resulting surface is shown in Figure \ref{fig:expesvi_simp}.
We produce prices using the Black-Scholes formula.

We select five times in which we will calibrate the model on generated prices : $t \in \{0.2, 0.4, 0.6, 0.8, 1.0\}$.
On time $\tau_i$, we select the calls with strikes $K \in \{S_0 + 1 + 4k_i\}$ and the puts with strikes $K \in \{S_0 - 1 - 4k_i\}$ for $k_i \in \{0, 1, \dots, N_{C,i}\}$, with $N_{C} = (5, 7, 9, 10, 12)$. We calibrate
less points for the earlier maturities as mass is almost nonexistent far from the at-the-money price at these maturities.

We perform the multiscale strategy described in Section \ref{sec:multiscale} with up to $N_T = 81$.

We choose the constant $c_{\text{mart}} = 1\times 10^4$ as the penalization term for the martingale constraint.

We implement the algorithm in Python using the PyKeops library.

The program runs in approximately 10 minutes on a V100 GPU with 24GB of GDDR5 memory and an Intel Xeon 5217 8 core CPU with 192GB or DDR4 RAM.

We show the convergence curves at the last scale $N_T=81$ in Figure \ref{fig:expe2_conv}.  Figure \ref{fig:expesvi_itererr} shows the
$L^2$ norm of the relative iterate errors $\frac{\|\Phi^{k+1} - \Phi^k\|}{\|\Phi^k\|}$ at each iteration.
 Figure \ref{fig:expesvi_marterr} shows the $L^2$ norm of the martingale error at each iteration.
 Figure \ref{fig:expesvi_priceerr} shows the $L^2$ norm of the price errors $\sum_i \|c_i - \E(G(X_{\tau_i}))\|^2$ by iteration for every scale. Different scales are separated by a dashed line.

Finally, Figure \ref{fig:expe2_cal} shows the calibration results at each time. For each of the calibration times,
we show the reference implied volatility, the calibrated implied volatility, and the implied volatility generated by
the forward diffusion process with the same number of timesteps and the volatility of the solution.
\begin{figure}[ht]
    \centering
    \includegraphics[width=0.7\textwidth]{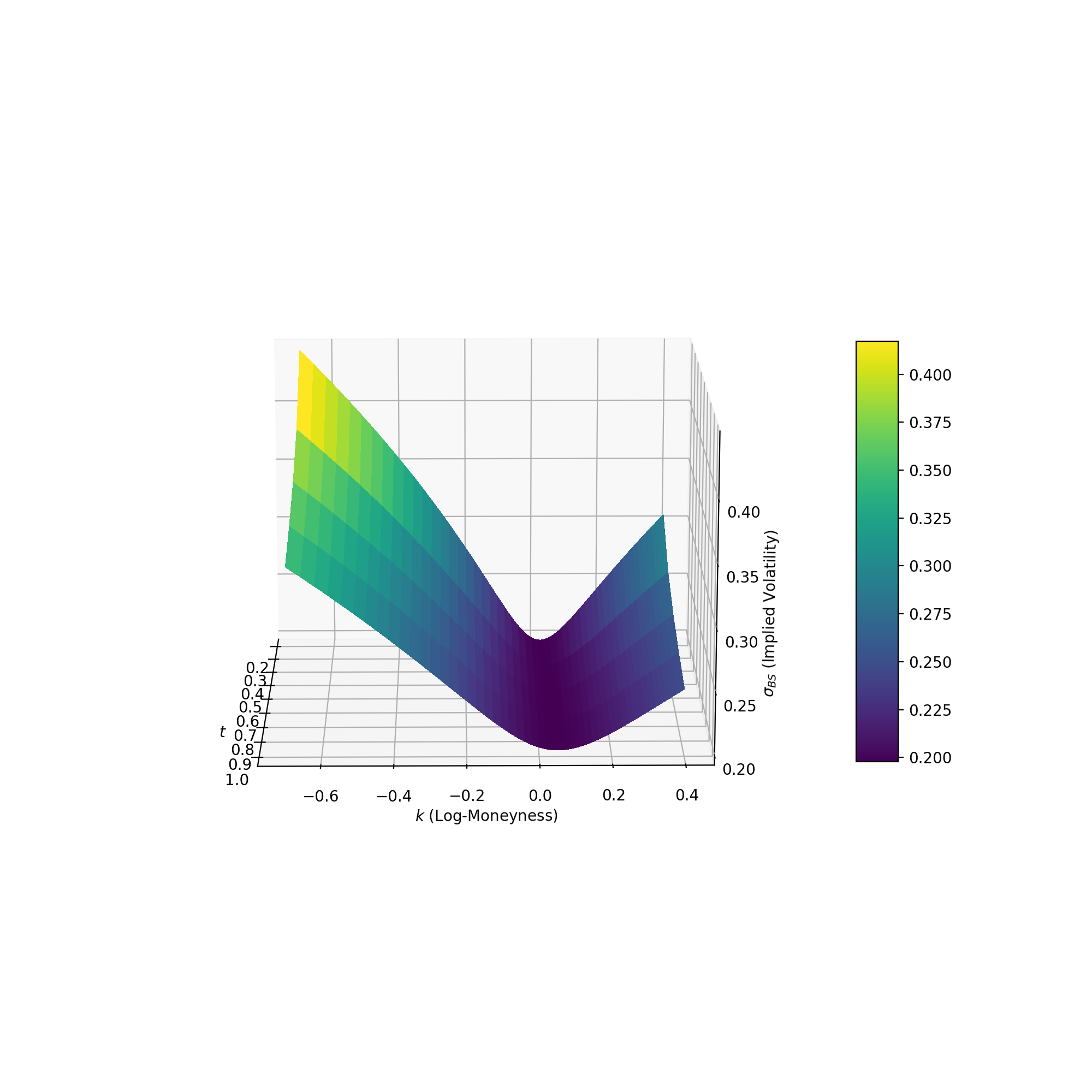}
    \caption{Generating model implied volatility}
    \label{fig:expesvi_simp}
\end{figure}

\begin{figure}[ht]
    \centering
    \begin{subfigure}[t]{0.45\textwidth}
        \includegraphics[width=\textwidth]{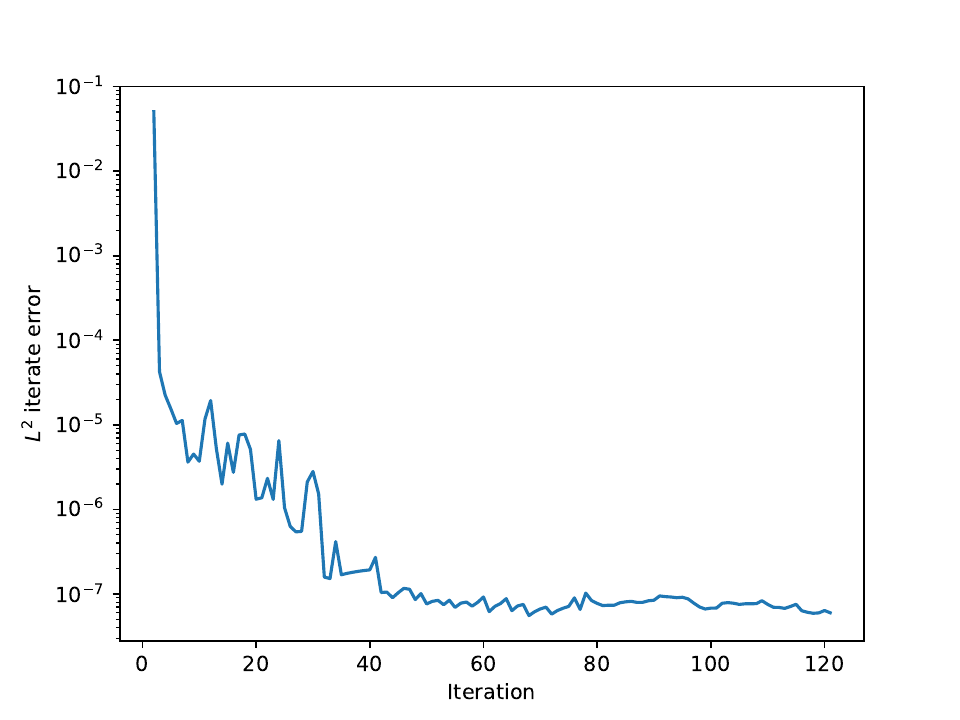}
        \caption{$L^2$ norm of iterate errors}
        \label{fig:expesvi_itererr}
    \end{subfigure}
    \begin{subfigure}[t]{0.45\textwidth}
        \includegraphics[width=\textwidth]{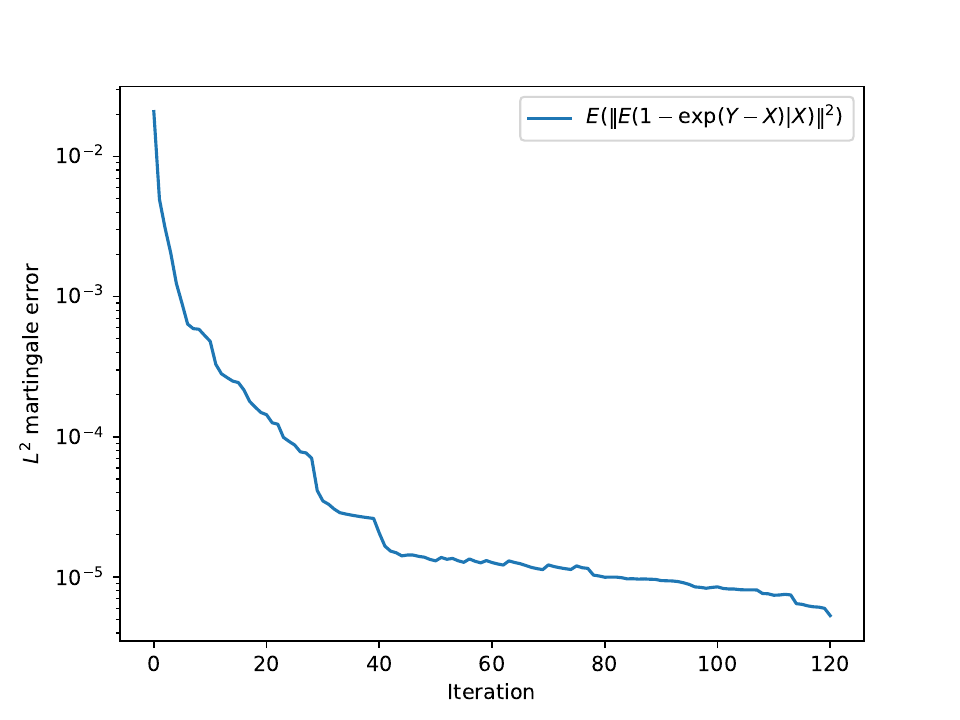}
        \caption{$L^2$ norm of the martingale error}
        \label{fig:expesvi_marterr}
    \end{subfigure}
    \begin{subfigure}[t]{0.45\textwidth}
        \includegraphics[width=\textwidth]{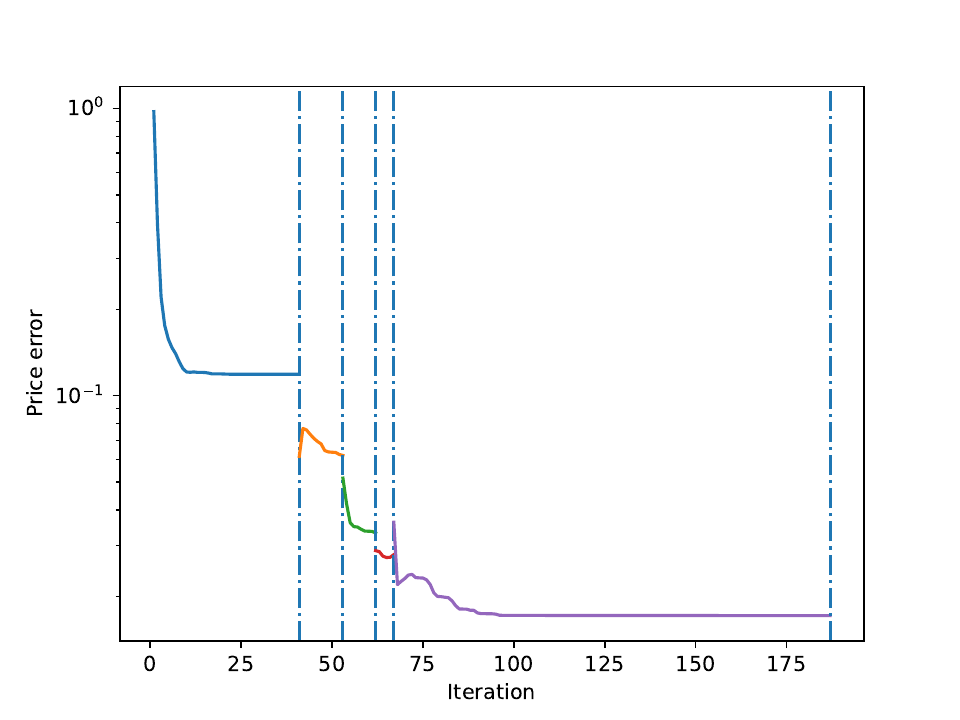}
        \caption{$L^2$ of the price errors}
        \label{fig:expesvi_priceerr}
    \end{subfigure}
    \caption{Convergence curves}
    \label{fig:expe2_conv}
\end{figure}

\begin{figure}[ht]
    \centering
    \begin{subfigure}[t]{0.45\textwidth}
        \includegraphics[width=\textwidth]{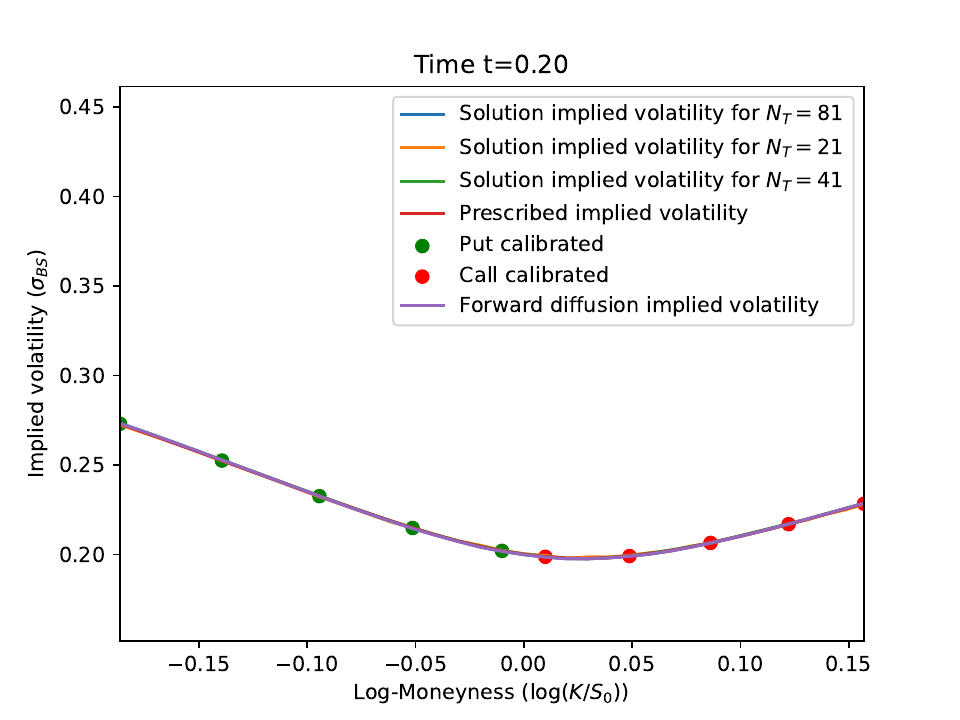}
        \caption{Calibration at time 0.2.}
        \label{fig:expesvi_cal02}
    \end{subfigure}
    \begin{subfigure}[t]{0.45\textwidth}
        \includegraphics[width=\textwidth]{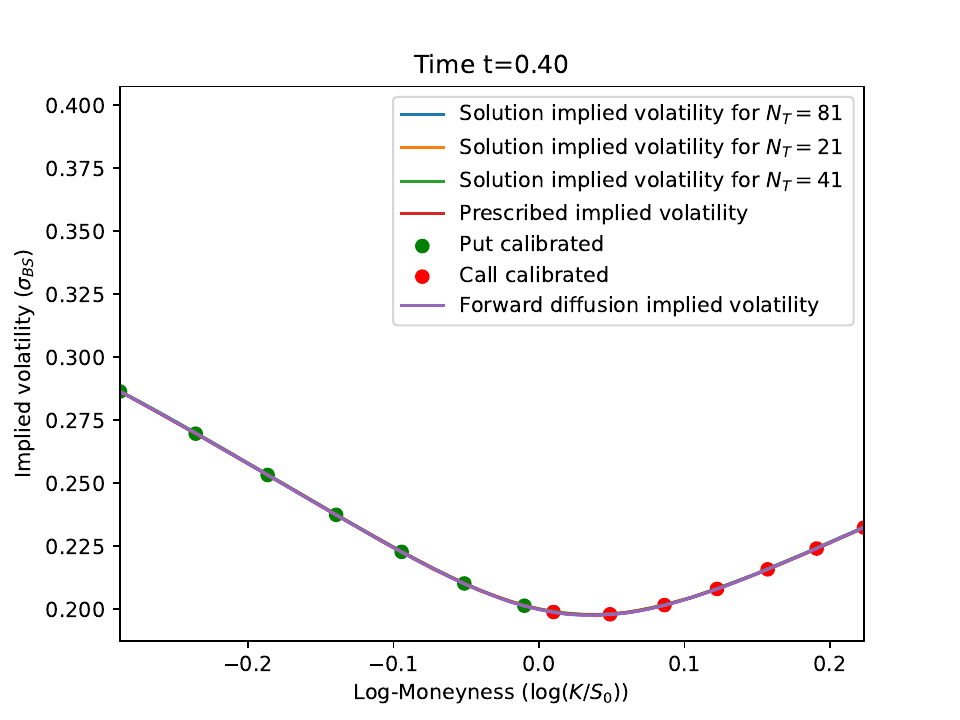}
        \caption{Calibration at time 0.4.}
        \label{fig:expesvi_cal04}
    \end{subfigure}
    \begin{subfigure}[t]{0.45\textwidth}
        \includegraphics[width=\textwidth]{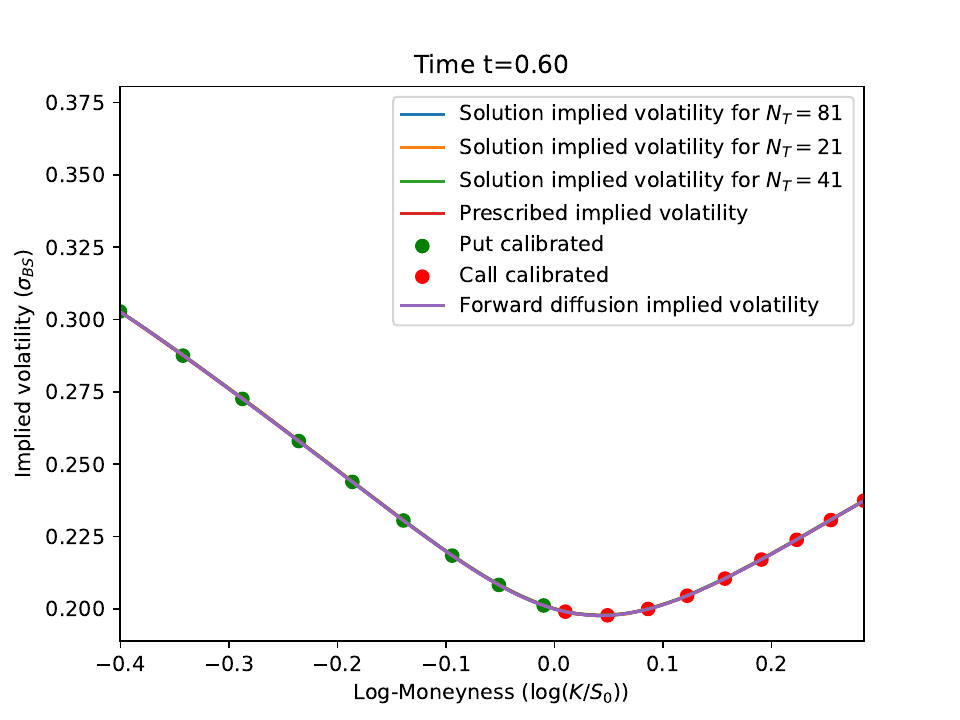}
        \caption{Calibration at time 0.6.}
        \label{fig:expesvi_cal06}
    \end{subfigure}
    \begin{subfigure}[t]{0.45\textwidth}
        \includegraphics[width=\textwidth]{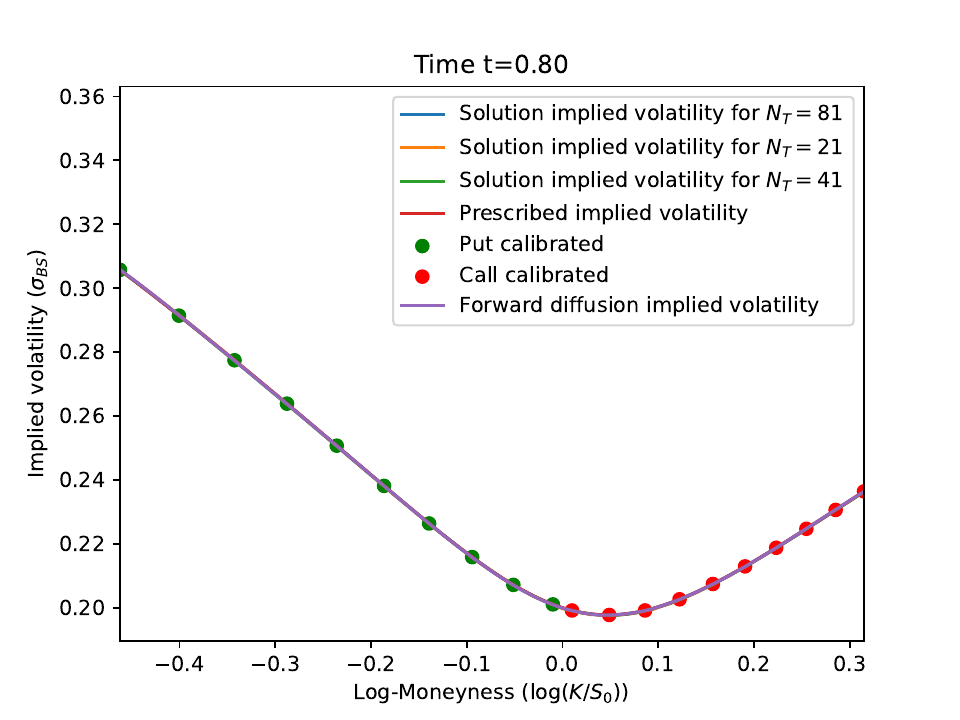}
        \caption{Calibration at time 0.8.}
        \label{fig:expesvi_cal08}
    \end{subfigure}
    \begin{subfigure}[t]{0.45\textwidth}
        \includegraphics[width=\textwidth]{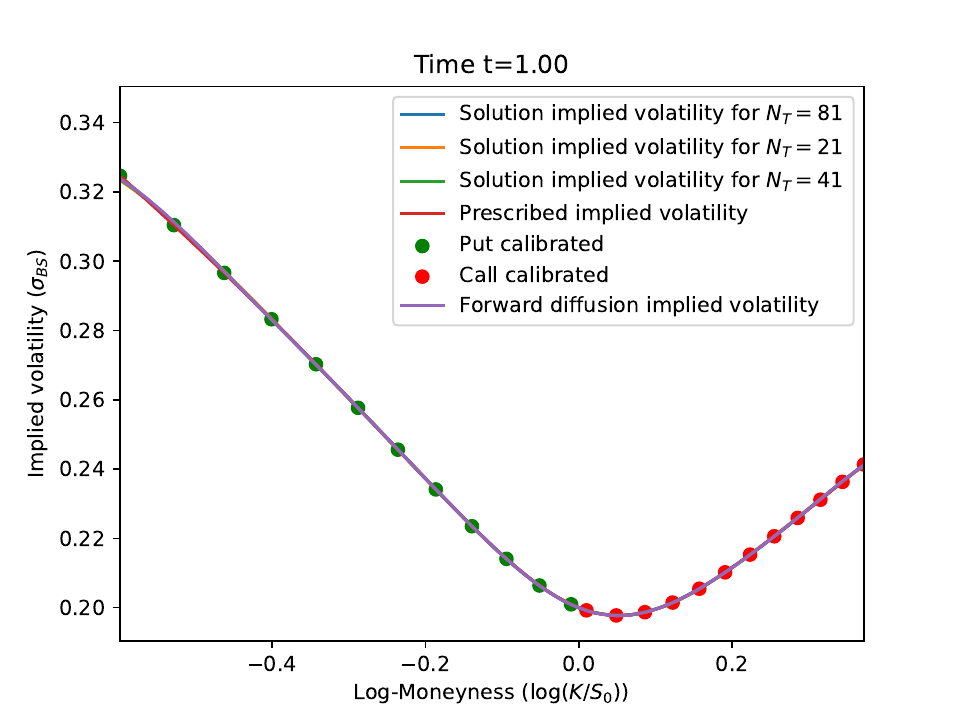}
        \caption{Calibration at time 1.}
        \label{fig:expesvi_cal10}
    \end{subfigure}
    \caption{Calibration results.}
    \label{fig:expe2_cal}
\end{figure}

%% file: appendix.tex
\subsection{Proof of Proposition \ref{prop_markovian}}
\label{proof_propmarkov}
First, we separate the sum of $\lambda_{g_i}$ per timesteps using the values defined above :
\begin{equation*}
    \sum_{i=0}^{N_C} \lambda_{g_i} G_i(x_i) = \sum_{k=0}^{N_T} \sum_{i=0}^{N_C} \lambda_{g_i} \mathds{1}_{\tau_i = k} G_i(x_i) = \sum_{k=0}^{N_T} \vv{\Lambda}_k \cdot \vv{G_k}(x_k)
\end{equation*}

We can rewrite the operator $\Delta$ as a sum :
\begin{align*}
    \Delta(\phi_{m}, \phi_{b}, \lambda_g) = & \sum_{k=0}^{N_T - 1} \Delta_{k,k+1}(x_k, x_{k+1}) + \vv{\Lambda_k} \cdot \vv{G_k} + \phi_{m_k}(x_k) \\
                                            & + \phi_{m_{N_T}}(x_{N_T}) + \vv{\Lambda_{N_T}} \cdot \vv{G_{N_T}}(x_{N_T})
\end{align*}
where only consecutive timesteps are grouped together. In particular, for a given $k$, we can separate this sum into three parts :
\begin{align*}
    \Delta(\phi_{m}, \phi_{p}, \phi_{d}, \lambda_g) = & \, \Delta_{k,k+1}(x_k, x_{k+1}) + \vv{\Lambda_k} \cdot \vv{G_k}(x_k) + \phi_{m_k}(x_k) \\
                                                      & + \vv{\Lambda_{k+1}} \cdot \vv{G_{k+1}}(x_{k+1}) + \phi_{m_{k+1}}(x_{k+1})             \\
                                                      & + \Delta^u_k(x_k) + \Delta^d_{k+1}(x_{k+1})
\end{align*}
where $\Delta^u_k$ and $\Delta^d_k$ are given by :
\begin{align*}
    \Delta^u_k(x_k) & = \sum_{i=0}^{k-1} \Delta_{i,i+1}(x_i, x_{i+1}) + \vv{\Lambda_i} \cdot \vv{G_i}(x_i) + \phi_{m_i}(x_i)                    \\
    \Delta^d_k(x_k) & = \sum_{i=k}^{N_T-1} \Delta_{i,i+1}(x_i, x_{i+1}) + \vv{\Lambda_{i+1}} \cdot \vv{G_{i+1}}(x_i) + \phi_{m_{i+1}}(x_{i+1}).
\end{align*}
We further note :
\begin{align*}
    \overline{\Delta}_{k, k+1}(x_k, x_{k+1}) = & \, \Delta_{k,k+1}(x_k, x_{k+1}) + \vv{\Lambda_k} \cdot \vv{G_k}(x_k) + \phi_{m_k}(x_k) \\
                                               & + \vv{\Lambda_{k+1}} \cdot \vv{G_{k+1}}(x_{k+1}) + \phi_{m_{k+1}}(x_{k+1})             \\
\end{align*}
for simplicity.

Given that $\bPP^h$ is separable in the same fashion, we can compute the joint probability between steps $k$ and $k+1$ as :
\begin{align*}
    \PP^{h,\star}_{k, k+1}(x_k, x_{k+1}) & = \int \PP^{h,\star}(dx_{[0,k-1]}, x_k, x_{k+1}, dx_{[k+2, N_T]})                                                                              \\
                                         & = \int e^{(\Delta^u_k + \overline{\Delta}_{k, k+1} + \Delta^d_{k+1})/h} \rho_0 \prod_{i=0}^{N_T} \bPP_{i, i+1}^h dx_{[0, k-1]} dx_{[k+2, N_T]} \\
                                         & = \left(\int e^{\Delta^u_k/h} \rho_0 \prod_{i=0}^{k-1} \bPP_{i, i+1}^h dx_{[0, k-1]}\right)                                                    \\
                                         & \hphantom{=} \times e^{\overline{\Delta}_{k, k+1}/h} \bPP^h_{k, k+1}                                                                           \\
                                         & \hphantom{=} \times \left(\int e^{\Delta^d_{k+1}/h} \prod_{i=k+1}^{N_T} \bPP_{i, i+1}^h dx_{[k+2, N_T]}\right)                                 \\
                                         & = \exp(\psi^u_k(x_k) + \overline{\Delta}_{k, k+1}(x_k, x_{k+1})/h + \psi^d_{k+1}(x_{k+1})) \bPP_{k, k+1}^h(x_k, x_{k+1})
\end{align*}

Similarly as in the previous proof, we can compute the marginal as:
\begin{align*}
    \PP^{h,\star}_k(x_k) & = \int \PP^{h,\star}(dx_{-k}, x_k)                                                                                                                                      \\
                         & = \int e^{(\Delta^u_k + \phi_{m_k}(x_k) + \vv{\Lambda_k} \cdot \vv{G_k}(x_k)  + \Delta^d_{k})/h} \rho_0 \prod_{i=0}^{N_T} \bPP_{i, i+1}^h dx_{[0, k-1]} dx_{[k+2, N_T]} \\
                         & = \left(\int e^{\Delta^u_k/h} \rho_0 \prod_{i=0}^{k-1} \bPP_{i, i+1}^h dx_{[0, k-1]}\right)                                                                             \\
                         & \hphantom{=} \times e^{\phi_{m_k}(x_k) + \vv{\Lambda_k} \cdot \vv{G_k}(x_k)/h}                                                                                          \\
                         & \hphantom{=} \times \left(\int e^{\Delta^d_{k}/h} \prod_{i=k}^{N_T} \bPP_{i, i+1}^h dx_{[k+1, N_T]}\right)                                                              \\
                         & = \exp(\psi^u_k(x_k) + \phi_{m_k}(x_k) + \vv{\Lambda_k} \cdot \vv{G_k}(x_k) + \psi^d_k(x_k))
\end{align*}